\documentclass[12pt]{amsart}

\usepackage{amsthm, amsmath, amssymb, amsrefs}

\usepackage{color}

\usepackage[single]{accents}

\newcommand{\T}{\mathbb{T}} 
\newcommand{\C}{\mathbb{C}} 
\newcommand{\D}{\mathbb{D}} 
\newcommand{\Z}{\mathbb{Z}}

\newcommand{\E}{\mathbb{E}}
\newcommand{\cD}{\overline{\D}} 

\newcommand{\ip}[2]{\langle #1, #2 \rangle}
\newcommand{\mcP}{\mathcal{P}}

\newcommand{\mcT}{\mathcal{T}}

\newcommand{\rev}[1]{\accentset{\leftarrow}{#1}}
\newcommand{\mcK}{\mathcal{K}}
\newcommand{\mcL}{\mathcal{L}}
\newcommand{\mcE}{\mathcal{E}}
\newcommand{\mcF}{\mathcal{F}}
\newcommand{\mcH}{\mathcal{H}}

\newcommand{\eAR}{\mathrm{eAR}}

\newcommand{\Cset}{\mathbb{C}}
\newcommand{\Nset}{\mathbb{N}}

\newcommand{\Zset}{\mathbb{Z}}

\numberwithin{equation}{section}

\newtheorem{theorem}{Theorem}[section]
\newtheorem{prop}[theorem]{Proposition}
\newtheorem{corollary}[theorem]{Corollary}
\newtheorem{lemma}[theorem]{Lemma}

\theoremstyle{definition}

\newtheorem{definition}[theorem]{Definition}

\newtheorem{remark}[theorem]{Remark}

\title{Polynomials with no zeros on a face of the
  bidisk}

\author[J.~S.~Geronimo]{Jeffrey~S.~Geronimo}
\thanks{JG partially supported by Simons Foundation Grant \#210169.}
\address{JG, School of Mathematics, Georgia Institute of Technology,
Atlanta, GA 30332--0160, USA}
\email{geronimo@math.gatech.edu}

\author[P.~Iliev]{Plamen~Iliev}
\thanks{PI partially supported by NSF grant DMS-0901092.}
\address{PI, School of Mathematics, Georgia Institute of Technology,
Atlanta, GA 30332--0160, USA}
\email{iliev@math.gatech.edu}

\author[G.~Knese]{Greg~Knese}
\address{GK, University of Alabama\\ Department of Mathematics
  \\ Tuscaloosa, AL 35487-0350}
\email{geknese@bama.ua.edu}

\thanks{GK partially supported by NSF grant DMS-1048775}

\date{January 14, 2013}

\begin{document}
\maketitle
\tableofcontents
\section{Introduction}\label{INTR}
 This article is concerned with harmonic analysis and
  moment problems as motivated by prediction theory and connections to
  analytic function theory and operator theory, continuing a tradition
  of classic works such as Helson-Lowdenslager \cite{HL58},
  Helson-Szeg\H{o} \cite{HS60}, and Wiener-Masani \cite{WM57}.  Most
  of these works are concerned with harmonic analysis on the unit
  circle and function theory on the unit disk.  In this article, we
  work in the setting of the two-torus or
  bi-circle. Helson-Lowdenslager \cite{HL58} was perhaps the first
  paper to pin down which aspects of harmonic analysis on the circle
  extend in a straightforward way to the bi-circle.  Factorization of
  positive trigonometric polynomials is one area that most certainly
  does not extend in a straightforward way from one variable to two,
  and this topic serves as a good starting point to motivate the rest
  of the paper.

The classical Fej\'er-Riesz lemma states that a non-negative
trigonometric polynomial $t(\theta)$ in one variable can be factored
as
\[
t(\theta) = |p(e^{i\theta})|^2
\]
where $p \in \C[z]$ is a polynomial with no zeros in the unit disk $\D
= \{z: |z|<1\}$.  While this is one of the simplest factorization
results it is useful in signal processing, trigonometric moment
problems, and wavelets.  It is also a prototype for more advanced and
important factorization results, such as Szeg\H{o}'s theorem. A simple
degrees of freedom argument shows that this result cannot be extended
without conditions to two variables. In recent years, progress has
been made in extending this result to two variables. First in
Geronimo-Woerdeman \cite{GW04} a characterization was given of
positive bivariate trigonometric polynomials $t$ that can be factored
as
\begin{equation} \label{tfactored}
t(\theta,\phi) = |p(e^{i\theta},e^{i\phi})|^2
\end{equation}
where $p\in \C[z,w]$ is stable, i.e.\ has no zeros in the closed
bidisk $\cD^2 = \cD\times \cD$.  The characterization is in terms of
trigonometric moments of the measure
\[
\frac{d\theta d\phi}{(2\pi)^2 t(\theta,\phi)} 
\]
on $[0,2\pi]^2$, and the necessary and sufficient conditions for the
characterization come from studying measures on $\T^2 = (\partial
\D)\times (\partial \D)$ of the form
\[
\frac{|dz||dw|}{(2\pi)^2 |p(z,w)|^2} 
\]
where $p$ is a polynomial. These are called Bernstein-Szeg\H{o}
measures. In one variable, such measures play a natural
  role since they can be used to match a finite sequence of moments of
  a given positive Borel measure in an ``entropy maximizing''
  way---see Landau \cite{hL87} or Simon \cite{bS05}.

Surprisingly, the development of the above result passes through a
sums of squares formula related to $p$ which yields a famous
inequality of And\^{o} from multivariable operator theory and in turn
yields Agler's Pick interpolation theorem for bounded analytic
functions on the bidisk. This connection is described in Cole-Wermer
\cite{CW99} and Knese \cite{gK08}.  The Hilbert space geometry
approach of \cite{gK08} made it possible to extend the
characterization to the setting where $p$ has no zeros on the open
bidisk $\D^2$ in \cite{gKAPDE}.  

In another direction, an investigation was begun in \cite{GW07} of
orthogonal polynomials associated with bivariate measures supported on
the bicircle constructed using the lexicographical or reverse
lexicographical ordering and the recurrence formulas associated with
these polynomials were developed. As in the one variable case a
spectral theory type result was proved relating the vanishing of
certain coefficients in the recurrence formulas to the existence of a
Fej\'er-Riesz type factorization and of a Bernstein-Szeg\H{o} measure
with $p$ a stable polynomial. Recently in \cite{GI12}, this viewpoint
yielded extensions of the above results to the problem of
characterizing positive bivariate trigonometric polynomials that can
be factored as in \eqref{tfactored} where now $p \in \C[z,w]$ has no
zeros on a \emph{closed face of the bidisk}. A closed face of the
bidisk refers to either $\T\times \cD$ or $\cD \times \T$.  This
result is significantly more difficult because much of the analyticity
of $1/p$ is lost. However the moments can still be computed using the
 one variable residue theorem. 
 Furthermore the factorization is in general not of the
 Helson-Lowdenslager type \cite{HL58} which would give a rational
 function rather than a polynomial factorization. Special
 consideration was given when the trigonometric polynomial
 $t(\theta,\phi)=|p(z,w)|^2=|q(z,w)|^2$ where $p(z,w)\ne 0$ for
 $|z|=1,\ |w|\le 1$ whereas $q(z,w)\ne 0$ for $|w|=1,\ |z|\le 1$, for
 which a spectral theory result analogous to the characterization of
 the Bernstein-Szeg\H{o} measures on the circle was shown to hold.

 In this article we refine, extend, and give a more complete picture
 of the results in \cite{GI12}; the case where $p$ has no zeros on a
 closed face of the bidisk mentioned above. In particular we emphasize
 that positive linear forms $\mcT$ on bivariate Laurent polynomials of
 bounded degree which can be represented as a Bernstein-Szeg\H{o}
 measure as above with $p(z,w)\ne 0$ for $|z|=1,\ |w|\le1$ can be
 characterized in two different ways by using $\mcT$ to define an
 inner product on polynomials: (1) a matrix condition involving
 certain natural truncated shift operators, and (2) the existence of a
 special orthogonal decomposition of spaces of polynomials we call
 {\em the split-shift orthogonality condition}.  The ``matrix
 condition'' is easier to verify (and matches the condition presented
 in \cite{GI12} when we choose an appropriate basis) while the
 ``split-shift orthogonality condition'' provides more information
 about the geometry of the spaces involved as well as the polynomial
 $p$.  In particular, this latter condition is key to proving a
 generalization of the sums of squares formula alluded to above.  A
 subtle fact is that the spaces involved in the split-shift condition
 are in general not unique but are in one-to-one correspondence with
 Fej\'er-Riesz type factorizations of the positive trigonometric
 polynomial $t$.  Furthermore given $p$ we present an explicit
 description of these spaces in terms of the decomposition of $p(z,0)$
 as a product of stable and unstable factors.  This makes it possible
 to characterize a whole stratification of factorizations of $t$ as
 $|p|^2$ where $p$ has no zeros on $\T \times \cD$ and $p(z,0)$ has a
 specified number of zeros in $\D$.  The case where $p(z,0)$ has no
 zeros in $\D$ recovers the Geronimo-Woerdeman characterization
 result, and the case where $p(z,0)$ has all zeros in $\D$ results in
 a related characterization of when $t=|p|^2$ where $p$ has no zeros
 in $(\C\setminus \D)\times \cD$.  In between these two extremes we
 can characterize when $t=|p|^2$ where the zero set of $p$ in $\D
 \times \C$ has a specified number of sheets over $z \in \D$ sitting
 in $\D\times \D$ (and a complementary number of sheets sitting in $\D
 \times (\C\setminus \cD)$).

We proceed as follows. In Section \ref{NSR} we introduce the notation
used throughout the paper and state the main theorems. In Section
\ref{basic} we derive basic orthogonality relations associated with
Bernstein-Szeg\H{o} measures.  In Section \ref{BSCiSSC} we show that a
Bernstein-Szeg\H{o} measure with $p(z,w)\ne 0$ for $|z|=1,\ |w|\le1$
implies the split-shift condition using a decomposition of $p(z,0)$
into stable and unstable factors.  In Section \ref{sec:SStoBS} we show
that the split-shift condition implies the existence of a
Bernstein-Szeg\H{o} measure of the type given above. Next in Section
\ref{se:MC} the matrix condition mentioned above is shown to be
equivalent to the split-shift condition.
In Section \ref{shiftsplits} we describe all $p$ that give rise to the
same positive bivariate trigonometric polynomial. In Section
\ref{constructp} we show how to construct $p$ from the moments
associated with the positive linear form. In Section
\ref{applications} we apply the previous results to solve the problem
when an extended bivariate autoregressive model has a causal or
acausal solution. Also in this section we give necessary and
sufficient conditions in terms of moments when a bivariate Borel
measure supported on the bicircle is a Bernstein-Szeg\H{o} measure
with $p$ nonzero for $|z|=1,\ |w|\le 1$.  Finally in Section \ref{GDV}
we adapt ideas from \cite{aK89} to give a second proof of our
generalized sum of squares formula which should be of independent
interest, while we also consider ``generalized distinguished
varieties'' and apply an argument of adapted from \cite{gKdv} to
obtain a sum of squares formula for polynomials associated with these
varieties. This allows us to obtain a determinantal representation of
the polynomial giving rise to the variety. \emph{Distinguished
  varieties} were introduced in \cite{AMdv} and play an important role
in multivariable operator theory and function theory on the
bidisk. Our determinantal representation generalizes one of the main
theorems of \cite{AMdv}.

\section{Notation and statement of results}\label{NSR}
We denote spaces of Laurent polynomials by
$$
\mcL_{j,k} = \vee\{z^sw^t:-j\leq s \leq j, -k\leq t \leq k\},
$$ 
where $\vee$ denotes the complex linear span of a set, and we
denote spaces of polynomials by
$$ 
\mcP_{j,k} := \vee\{z^sw^t:0\leq s \leq j, 0\leq t \leq k\},
$$ 
where $j,k\in\Z_{+}$. In some parts of the paper, the spaces
$\mcL_{j,k}$ and $\mcP_{j,k}$ appear naturally within the context of
the Hilbert space $L^{2}(\T^2,\mu)$, where $\mu$ is a positive Borel
measure on $\T^2$, in which case we may use also $j,k=\infty$ by
considering the closed linear spans above.

A linear form $\mcT:\mcL_{n,m}\rightarrow\C$ is said to be
\emph{positive} if
\[
\mcT(f(z,w)\bar{f}(1/z,1/w)) > 0
\]
for every nonzero $f \in \mcP_{n,m}$, where $\bar{f}(z,w) =
\overline{f(\bar{z},\bar{w})}$. With $\mcT$ we define an inner product
on the space $\mcP_{n,m}$, via
\[
\ip{f}{g}_{\mcT} = \mcT(f(z,w)\bar{g}(1/z,1/w)) \qquad f,g \in \mcP_{n,m}.
\]
Let us write $\mcH_{\mcT}$ for the finite dimensional Hilbert space
$(\mcP_{n,m}, \ip{\cdot}{\cdot}_{\mcT})$.

For $(k,l)\in\Z_{+}^2$ where the inner product above is defined 
we denote the following orthogonal complements:
\begin{align}
\mcE_{k,l}^1 &= \mcP_{k,l} \ominus w \mcP_{k,l-1},\label{E1}\\
\mcF_{k,l}^1 &= \mcP_{k,l} \ominus \mcP_{k,l-1},\label{F1}\\
\mcE_{k,l}^2 &= \mcP_{k,l} \ominus z \mcP_{k-1,l},\label{E2}\\
\mcF_{k,l}^2 &= \mcP_{k,l} \ominus  \mcP_{k-1,l}.\label{F2}
\end{align}

We will often employ the anti-unitary reflection operator
$\rev{\cdot}$ 
$$g(z,w) \mapsto \rev{g}(z,w):= z^{k}w^l\bar{g}(1/z,1/w)$$ which in
this case we say is applied at the degree $(k,l)$.  This degree will
usually be clear from context or explicitly stated.  For example,
applying this operator at degree $(k,l)$ to the spaces $\mcF_{k,l}^1$
and $\mcF_{k,l}^2$ we see that
\[
\mcE_{k,l}^1 = \rev{\mcF}_{k,l}^1 \text{ and } \mcE_{k,l}^2 =
\rev{\mcF}_{k,l}^2,
\]
since the operator is an anti-unitary in $\mcH_{\mcT}$.

\begin{definition}\label{SHC}
A positive linear form $\mcT$ on $\mcL_{n,m}$ satisfies the
\emph{split-shift orthogonality condition} if there exist subspaces of
polynomials $\mcK_1, \mcK_2 \subset \mcH_{\mcT}$ such that
\begin{enumerate}
\item $\mcE_{n-1,m}^1 = \mcK_1\oplus \mcK_2$
\item $\mcK_1 \perp z\mcK_2$ and 
\item $\mcK_1, z\mcK_2 \subset \mcE_{n,m}^1$.
\end{enumerate}
\end{definition}

The point of conditions (2) and (3) is that they imply $\mcK_1 \oplus
z\mcK_2 \subset \mcE_{n,m}^1$.  This condition actually characterizes
positive linear forms coming from a Bernstein-Szeg\H{o} measure.  What
is interesting is that this condition can also be expressed using a
simple matrix condition.

To present the matrix condition let us define three operators
\[
\begin{aligned}
A &= P_{w\mcE_{n,m-1}^2} M_z : \mcE_{n-1,m}^1 \to w\mcE_{n,m-1}^2 \\
B &= P_{\mcE_{n-1,m}^1}: w \mcF_{n,m-1}^2 \to \mcE_{n-1,m}^1\\
T &= P_{\mcE_{n-1,m}^1} M_z : \mcE_{n-1,m}^1 \to \mcE_{n-1,m}^1
\end{aligned}
\]
where $M_z$ is multiplication by $z$ and $P_{\mcH}$ represents
orthogonal projection onto a subspace $\mcH \subset \mcH_{\mcT}$.
Notice that $T$ is just truncation of multiplication by $z$ to
$\mcE_{n-1,m}^1$. 

\begin{theorem} \label{mainthm}
Let $\mcT$ be a positive linear form on $\mcL_{n,m}$.  The following
are equivalent.
\begin{enumerate}
\item (Bernstein-Szeg\H{o} condition) There exists $p \in \C[z,w]$
  with no zeros on $\T\times \cD$ and degree at most $(n,m)$ such that
\begin{equation}\label{BSmoments}
\mcT(z^jw^k) = \int_{\T^2} z^jw^k
\frac{|dz||dw|}{(2\pi)^2|p(z,w)|^2} \qquad |j|\leq n, |k| \leq m.
\end{equation}

\item (Split-shift condition) $\mcT$ satisfies the split-shift
  orthogonality condition.

\item (Matrix condition) The invariant subspace of $T$ generated by
  the range of $B$ is contained in the kernel of $A$.  More
  concretely,
\begin{equation} \label{mc}
AT^j B = 0 \text{ for } j = 0,1,\dots, n-1.
\end{equation}
\end{enumerate}

\end{theorem}

This theorem is a more geometric formulation of the results in
Geronimo-Iliev \cite{GI12}.  In particular, the coordinate free
formulation of condition (3) makes it possible to give a
straightforward proof of the equivalence of (2) and (3) in
Propositions \ref{sstomc} and \ref{mctoss}---the original proof in
\cite{GI12} involves some non-trivial linear algebra.  Of greater
significance, however, is \emph{our emphasis on the split-shift
  condition and the rather complete knowledge it provides of the
  geometry of Bernstein-Szeg\H{o} measures of the above type.}  A
version of the split-shift condition was recognized as an important
stepping stone in \cite{GI12}, but at that time it was not clear how
to construct the spaces involved directly with Hilbert space
geometry---this question was explicitly raised as \cite{GI12}*{Remark
  5.3}.
The approach developed here resolves this.

\begin{theorem} \label{bssplits}
Let $\mcT$ be a positive linear form on $\mcL_{n,m}$ satisfying the
Bernstein-Szeg\H{o} condition of Theorem \ref{mainthm} with polynomial
$p(z,w)$ having no zeros on $\T\times \cD$ and degree at most
$(n,m)$. Then
\[
\mcK_1 = P_{\mcE_{n-1,m}^1}\{g(z) a(z): g \in \C[z], \deg g < \deg b\}
\]
\[
\mcK_2 = P_{\mcE_{n-1,m}^1}\{g(z) b(z): g \in \C[z], \deg g < n-\deg b\},
\]
satisfy the split-shift condition. Here $p(z,0)=a(z)b(z)$ where $a\in
\C[z]$ has no zeros in $\cD$ and $b\in \C[z]$ has all zeros in $\D$.
\end{theorem}

More explicitly, if we form the span of the following projections of
one variable polynomials
\[
z^ia(z) - P_{w\mcP_{n-1,m}} z^i a(z) \text{ for } 0\leq i < \deg b
\]
and 
\[
z^i b(z) - P_{w\mcP_{n-1,m}} z^i b(z) \text{ for } 0 \leq i < n-\deg b
\]
then the resulting subspaces satisfy all the orthogonality conditions
in the split-shift definition. 

Why should we emphasize the abstract looking split-shift condition in
the first place?  One answer to this is that the spaces in the
split-shift condition appear naturally in the following sum of
(hermitian) squares result that ends up being an important by-product
of our work.

\begin{theorem} \label{sosthm}
 Suppose $p \in \C[z,w]$ has no zeros on $\T\times \cD$ and
  $\deg p = (n,m)$.  Define $\rev{p}(z,w) = z^nw^m
  \overline{p(1/\bar{z},1/\bar{w})}$.  Then, there exist polynomials
  $A_1,\dots, A_m, B_1,\dots, B_{n_1}, C_1,\dots, C_{n_2} \in \C[z,w]$
  such that
\[
\begin{aligned}
&|p(z,w)|^2 - |\rev{p}(z,w)|^2 \\
&= (1-|w|^2)\sum_{j=1}^{m} |A_j(z,w)|^2 +
  (1-|z|^2)\left(\sum_{j=1}^{n_1} |B_j(z,w)|^2 - \sum_{j=1}^{n_2}
  |C_j(z,w)|^2\right)
\end{aligned}
\]
where $n_2$ is the number of zeros of $p(z,0)$ in $\D$ and $n_1 = n -
n_2$.  The same result holds if $p$ has no zeros in $\T\times \D$ and
no factors in common with $\rev{p}$.
\end{theorem}

The different sums of squares terms can be constructed from important
subspaces of $L^2(\frac{|dz||dw|}{|p|^2} )$ : the $A_j$ form an
orthonormal basis of $\mcE_{n,m-1}^2$, the $B_j$ form an orthonormal
basis of $\rev{\mcK}_2$ (the reflection of $\mcK_2$), and the $C_j$
form an orthonormal basis of $\mcK_1$. (See Theorem
\ref{aglerdecomps}.)  This formula illustrates how natural are  the
spaces in the split-shift condition, and it also reproves some
important formulas as special cases.

When $n_2=0$, $p$ is stable and we get the Cole-Wermer type of sum of
squares formula \cite{CW99} which can be used to prove Agler's Pick
interpolation theorem on the bidisk; see also \cite{GW04},
\cite{gK08}, \cite{GIK12}, \cite{aK89}, \cite{BSV}. The exact numbers
of squares involved in this case turned out to be important in recent
work on extending L\"owner's theory of matrix monotone functions to
two variables in Agler-McCarthy-Young \cite{AMY}.  When $m=0$
(i.e. $p$ does not depend on $w$) we get a decomposition which readily
implies part of the Schur-Cohn method for counting the roots of a
polynomial inside and outside the unit circle.
 
The case of $p$ with merely no zeros on $\T\times \D$ can be derived
from a limiting argument as in \cite{gK08}.  We give a second proof of
the sum of squares formula using ideas of Kummert \cite{aK89} in
Section \ref{sec:kummert}.  This proof should be of independent
interest and has the advantage of working directly for all cases.  See
Section \ref{detreps} for an application of the formula to proving a
determinantal representation for a class of curves generalizing the
\emph{distinguished varieties} of Agler-McCarthy \cite{AMdv}.

Now that we see that the split-shift condition is natural, we get into
a deeper discussion of  Theorem \ref{mainthm}, and
its extensions.  From the maximum entropy principle there are at most
finitely many $p(z,w)$ for which the Bernstein-Szeg\H{o} condition
holds.  Theorem \ref{bssplits} therefore gives one particular way to
construct $\mcK_1, \mcK_2$ in the definition of split-shift and this
way is uniquely determined by the choice of $p$.  However, since a
trigonometric polynomial factored as $|p|^2$ can potentially be
factored in more than one such way---roughly speaking these
polynomials can be obtained from one another by permuting the factors
in $|p(z,w)|^2$ which depend only on $z$---each such factorization
will yield spaces as in the split-shift condition via the above
theorem.  We prove that these are all the possible split-shift
decompositions corresponding to $|p(z,w)|^2$.  See Proposition
\ref{allshiftsplits}.

While each choice of $p$ in the factorization of $t=|p|^2$ yields a
canonically associated pair of spaces $\mcK_1, \mcK_2$ in the
split-shift condition, the matrix condition naturally gives rise to
two canonical choices for such pairs.

\begin{theorem} \label{matrixsplits}
Let $\mcT$ be a positive linear form on $\mcL_{n,m}$ satisfying the
matrix condition of Theorem \ref{mainthm}.  Then, $(\mcK_1,\mcK_2) =
(\mcE_{n-1,m}^1\ominus \mathcal{B}, \mathcal{B})$ satisfies the
split-shift condition where
\[
\mathcal{B} = \vee \{ T^jB f: f \in w \mcF_{n,m-1}^2, j=0,1, \dots\}.
\]

Similarly, $(\mcK_1,\mcK_2) = (\mathcal{A},\mcE_{n-1,m}^1\ominus
\mathcal{A})$ satisfies the split-shift condition where
\[
\mathcal{A} = \vee \{(T^*)^{j} A^* f: f \in w \mcE^{2}_{n,m-1},
j=0,1,\dots\}.
\]
If $(\mcK_1',\mcK_2')$ is any other pair satisfying the split-shift
condition, then $\mathcal{A} \subset \mcK_1'$ and $\mathcal{B} \subset
\mcK_2'$.
\end{theorem}

To be clear, $T^*: \mcE^{1}_{n-1,m} \to \mcE^{1}_{n-1,m}$ is given by
$P_{\mcE^{1}_{n-1,m}} M_{1/z}$ and $A^*: w \mcE^{2}_{n,m-1} \to
\mcE^{1}_{n-1,m}$ is given by $P_{\mcE^{1}_{n-1,m}} M_{1/z}$.   See
Theorems \ref{ABdecomp} and \ref{BAformulas} where we also show how
the spaces in Theorem \ref{matrixsplits} relate to those in Theorem
\ref{bssplits}.   

Theorems \ref{bssplits} and \ref{matrixsplits} directly show how the
Bernstein-Szeg\H{o} condition and the matrix condition yield the
split-shift condition.  On the other hand, if the split-shift
condition holds, $\mcK_1\oplus z\mcK_2$ has co-dimension one in
$\mcE_{n,m}^1$ and we shall show that the Bernstein-Szeg\H{o}
condition holds using any unit norm element $p$ in the one dimensional
space $\mcE_{n,m}^1\ominus (\mcK_1\oplus z\mcK_2)$.  A sum of squares
result related to Theorem \ref{sosthm} ends up being crucial here.  In
Section \ref{constructp}, we describe a simple procedure for
constructing $p$ from the moments $\mcT(z^jw^k)$ once we know the
split-shift condition holds.

Our emphasis on the split-shift condition permits several interesting
refinements that were not evident before.  Notice that if $p$ does not
vanish on $\T\times \cD$, then the argument principle shows that the
number of zeros of $p(\cdot, w)$ in $\D$ will be constant as $w$
varies in $\cD$. Thus it is possible to prove a ``stratified'' version
of Theorem~\ref{mainthm}, where we characterize factorizations
involving $p$ with no zeros in $\T\times \cD$ such that $p(z,0)$ has a
specified number of zeros in $\D$.  See the end of Section
\ref{shiftsplits} for the proof of the following corollary.

\begin{corollary} \label{corstrat}
Let $\mcT$ be a positive linear form on $\mcL_{n,m}$ and
let $0\leq d \leq n$.  The following are equivalent.
\begin{enumerate}
\item (Bernstein-Szeg\H{o} condition) There exists $p \in \C[z,w]$
  with no zeros on $\T\times \cD$, degree at most $(n,m)$, and where
  $p(z,0)$ has $d$ zeros in $\D$ such that
\[
\mcT(z^jw^k) = \int_{\T^2} z^jw^k
\frac{|dz||dw|}{(2\pi)^2|p(z,w)|^2} \qquad |j|\leq n, |k| \leq m.
\]

\item (Split-shift condition) $\mcT$ satisfies the split-shift
  orthogonality condition where $\mcK_1$ has dimension $d$.

\item (Matrix condition) The invariant subspace of $T$ generated by
  the range of $B$ is contained in the kernel of $A$, and 
\[
\dim \mathcal{A} \leq d \leq n- \dim \mathcal{B}.
\]
Note $\mathcal{A}$ and $\mathcal{B}$ are as in Theorem
\ref{matrixsplits}.
\end{enumerate}

\end{corollary}

  In particular, the case $d=0$ yields the Geronimo-Woerdeman result
  (as well as much simpler looking conditions).  In this case, the
  split-shift condition merely says
\begin{equation} \label{GWss}
z\mcE^1_{n-1,m} \subset \mcE^1_{n,m}.
\end{equation}
The matrix condition in this case implies $\mathcal{A} = \{0\}$ which
implies $A = 0$.  Since the range of $A$ is $P_{w\mcE^2_{n,m-1}} z
\mcE^1_{n-1,m}$, this means $w\mcE^2_{n,m-1} \perp z\mcE^1_{n-1,m}$,
which is equivalent to \eqref{GWss} because of the orthogonal
decomposition
\[
\mcE^1_{n,m} \oplus w \mcE^{2}_{n,m-1} = z\mcE^{1}_{n-1,m} \oplus
\mcE^2_{n,m}.
\]
By performing the reflection operation, $w\mcE^2_{n,m-1} \perp
z\mcE^1_{n-1,m}$ is equivalent to $\mcF^{1}_{n-1,m} \perp
\mcF^{2}_{n,m-1}$.  
 
\begin{corollary}[Geronimo-Woerdeman \cite{GW04}]
Let $\mcT$ be a positive linear form on $\mcL_{n,m}$.  There exists $p
\in \C[z,w]$ with no zeros on $\cD^2$ and degree at most $(n,m)$ such
that
\[
\mcT(z^jw^k) = \int_{\T^2} z^jw^k
\frac{|dz||dw|}{(2\pi)^2|p(z,w)|^2} \qquad |j|\leq n, |k| \leq m,
\]
if and only if
\[
\mcF^{1}_{n-1,m} \perp
\mcF^{2}_{n,m-1}.
\]
\end{corollary}
To use the language of \cite{LL12}, the last condition can be neatly
phrased as saying $\mcP_{n-1,m}$ and $\mcP_{n,m-1}$ intersect at right
angles.

As in \cite{GI12}, Theorem \ref{mainthm} allows us to characterize
when a positive two variable trigonometric polynomial can be factored
as $|p(z,w)|^2$ on $\T^2$ where $p$ has no zeros in $\T\times \cD$.

\begin{theorem} \label{frthm}
Suppose $t(z,w) = \sum_{j=-n}^{n} \sum_{k=-m}^{m} t_{jk} z^j w^k > 0$
for $(z,w) \in \T^2$.  Then, there exists $p \in \C[z,w]$ of degree at
most $(n,m)$ with no zeros on $\T\times\cD$ such that $t = |p|^2$ on
$\T^2$ if and only if the positive linear form $\mcT$ on $\mcL_{n,m}$
\[
\mcT(z^jw^k) = \int_{\T^2} z^jw^k \frac{|dz||dw|}{(2\pi)^2 t(z,w)}
\qquad |j|\leq n, |k| \leq m
\]
satisfies the split-shift condition. 
\end{theorem}
See the end of Section \ref{sec:SStoBS} for a proof of this theorem. 

We say a finite, positive Borel measure $\mu$ on $\T^2$ is
\emph{non-degenerate} if
$$\int_{\T^2}|f|^2d\mu>0,$$ for every nonzero polynomial $f \in
\C[z,w]$.  We next turn to the problem of characterizing which such
measures $\mu$ on $\T^2$ are of the form
\begin{equation} \label{BSform}
\frac{1}{|p|^2} d\sigma
\end{equation}
where $p \in \C[z,w]$ has no zeros in $\T\times \cD$ and degree at
most $(n,m)$; $d\sigma$ denotes normalized Lebesgue measure on
$\T^2$. 
 
Some necessary conditions turn out to be
\begin{equation} \label{someneccond}
\mcE_{n,M}^2=\mcE^2_{n+j,M} 
\end{equation}
for $M\geq m-1$ and $j\geq 0$.  These conditions are most likely not
sufficient though.  

Surprisingly, in \cite{GI12}, it was noticed that conditions
\eqref{someneccond} combined with the analogous conditions obtained by
interchanging the roles of $z$ and $w$ characterize when $\mu$ has the
form
\[
\frac{1}{|p(z,w)q(1/z,w)|^2} d\sigma
\]
where $p,q \in \C[z,w]$ have no zeros in $\cD^2$.  

We now provide the following necessary and sufficient conditions for
$\mu$ to have the form \eqref{BSform}.  Define the following one
dimensional spaces 
\begin{equation}\label{HM}
\mcH_{M} := \mcP_{2n,M}\ominus \vee\{z^jw^k: 0\leq j\leq 2n, 0\leq k
\leq M, (j,k)\ne(n,0)\}.
\end{equation}

\begin{theorem}
Let $d\mu$ be a non-degenerate, finite, positive Borel measure on
$\T^2$.  There exists $p\in \C[z,w]$ of degree at most $(n,m)$ with no
zeros on $\T\times \cD$ such that
\[
d\mu = \frac{d\sigma}{|p|^2}
\]
if and only if 
\[
\mcE_{n,M}^2 = \mcE_{n+j,M}^2 
\text{ and } \mcH_{m} = \mcH_{m+j}
\]
for $M\geq m-1$ and $j\geq 0$.
\end{theorem}

This theorem is proved in Section \ref{fullmeasure}, and in
\ref{concretefullmeasure} it is expressed concretely in terms of the
moments of $\mu$.  In Section \ref{autoreg}, we discuss the close
connection of our main theorem, Theorem \ref{mainthm}, to
autoregressive filters as was done in \cite{GW04}.

\section{Basic orthogonalities of Bernstein-Szeg\H{o}
  measures} \label{basic} 
The next two sections are occupied with proving that the
Bernstein-Szeg\H{o} condition implies the split-shift condition in
Theorem \ref{mainthm}, which is the content of Theorem \ref{BStoSS}.
The approach is an extension of \cite{GIK12}.

Let $p \in \C[z,w]$ and assume $p(z,w)\ne 0$ for $(z,w) \in \T\times
\cD$.  Let $\deg p \leq (n,m)$, $\rev{p}(z,w) = z^n w^m
\overline{p(1/\bar{z},1/\bar{w})}$.  Let $d\sigma$ denote normalized
Lebesgue measure on $\T^2$.  We use $d\sigma(z) = |dz|/(2\pi)$ or
$d\sigma(w) = |dw|/(2\pi)$ to denote normalized Lebesgue measure on
$\T$ using the variable $z$ or $w$.  We use $\ip{\cdot}{\cdot}$ for
the inner product in $L^2(1/|p|^2d\sigma, \T^2)$ and $\vee$ to
denote closed linear span in both $L^2(\T^2)$ and
$L^2(1/|p|^2d\sigma)$. This is legitimate as $L^2(1/|p|^2d\sigma)$ is
homeomorphic to $L^2(\T^2)$ since $|p|$ is bounded above and below on
$\T^2$. The next lemma shows that $p$ and $\rev{p}$ are orthogonal to
all monomials in half planes.

\begin{lemma} \label{lem:p}
In $L^2(\frac{1}{|p|^2} d\sigma)$,
\[
p \perp z^jw^k 
\]
for $j\in \Z, k \geq 1$ and  
\[
\rev{p} \perp z^jw^k
\]
for $j \in \Z, k <m$.
Also,
\[
\vee\{z^jp: j \in \Z\} = \vee\{z^j w^k: j \in \Z, 0 \leq k \leq m\}
\ominus \vee\{z^j w^k: j \in \Z, 1 \leq k \leq m\}.
\]
\end{lemma}
\begin{proof}
\[
\ip{z^jw^k}{p} = \int_{\T} z^j \int_{\T}\frac{w^k}{p(z,w)} d\sigma(w) d\sigma(z) = 0
\]
for $k\geq 1$ since $1/p(z,w)$ is holomorphic in $w \in \cD$ when $z
\in \T$. The proof for $\rev{p}$ is similar.

For the final part, we have just shown the inclusion $\subset$.  On
the other hand, if $f \in \vee\{z^j w^k: j \in \Z, 0 \leq k \leq m\}
\ominus \vee\{z^j w^k: j \in \Z, 1 \leq k \leq m\}$ and $f \perp z^j
p$ for all $j\in \Z$, then
\[
0 = \int_{\T^2} \frac{f(z,w)z^{-j}}{p(z,w)} d\sigma = \int_{\T}
\frac{f(z,0)z^{-j}}{p(z,0)} d\sigma(z)
\]
for all $j \in \Z$ implies $f(z,0)/p(z,0) = 0$ for a.e. $z \in \T$.
(Note that $f(z,0)$ should be interpreted as $\sum_{j \in \Z}
\hat{f}(j,0)z^j$ in $L^2$.)  Therefore, $f(z,0) = 0$ which implies $f
\in \vee\{z^j w^k: j \in \Z, 1 \leq k \leq m\}$ making $f$ orthogonal
to itself.  So, $f=0$.
\end{proof}

Define 
\[
J_{\eta}(z,w) = z^n
\frac{p(z,w)\overline{p(1/\bar{z},\eta)}}{1-w\bar{\eta}}
\]
\[
H_{\eta}(z,w) = z^n \frac{\rev{p}(z,w)
  \overline{\rev{p}(1/\bar{z},\eta)}}{1-w\bar{\eta}}.
\]
By the previous lemma, $H_\eta \perp \vee\{z^jw^k: j \in \Z, k<m\}$
for $\eta \in \D$ and $J_{\eta} \perp \vee\{z^jw^k: j \in \Z, k\geq
0\}$ for $|\eta| > 1$ since for $(z,w) \in \T^2$
\[
J_{\eta}(z,w) = \frac{-\bar{w}}{\bar{\eta}} z^n
\frac{p(z,w)\overline{p(1/\bar{z},\eta)}}{1-\bar{w}/\bar{\eta}}.
\]

Define
\[
L_{\eta} (z,w) = L(z,w;\eta) = z^n
\frac{p(z,w)\overline{p(1/\bar{z},\eta)} - \rev{p}(z,w)
  \overline{\rev{p}(1/\bar{z},\eta)}}{1-w\bar{\eta}} = J_{\eta}(z,w) - H_{\eta}(z,w)
\]
which is a polynomial in $(z,w,\bar{\eta})$ of degree $(2n,m-1,m-1)$.  
Notice that
\begin{equation}\label{Lsymm}
z^{2n}(w\bar{\eta})^{m-1}
\overline{L(1/\bar{z},1/\bar{w};1/\bar{\eta})} = \bar{\eta}^{m-1}
\rev{L}_{1/\bar{\eta}}(z,w) = L_{\eta}(z,w).
\end{equation}

Similarly we define
\[
G_{\eta}(z,w) = G(z,w;\eta) = J_{\eta}(z,w) - w\bar{\eta}
H_{\eta}(z,w)
\]
which is a polynomial in $(z,w,\bar{\eta})$ of degree $(2n,m,m)$. Note
that the reflection symmetry for $L_{\eta}(z,w)$ implies the following
symmetry for $G_{\eta}(z,w)$
\begin{equation}\label{Gsymm}
z^{2n}(w\bar{\eta})^{m}
\overline{G_{1/\bar{\eta}}(1/\bar{z},1/\bar{w})} = G_{\eta}(z,w).
\end{equation}

Up to factors of $z^n$, $L_{\eta}(z,w)$ and $G_{\eta}(z,w)$ are
parametrized one-variable Christoffel-Darboux kernels. In the next few
lemmas, the orthogonality properties of $p$ and $\rev{p}$ are used to
obtain orthogonality properties on pieces of these kernels.

\begin{lemma}
If $f \in L^2$ and $\text{supp}(\hat{f}) \subset \Z\times \Z_{+}$,
then in $L^2(1/|p|^2 d\sigma)$
\[
\ip{f}{J_{\eta}} = \sum_{k\geq 0} \hat{f}(n,k) \eta^k
\]
for $\eta \in \D$.  In particular, $J_{\eta} \perp \vee \{z^j w^k: k
\geq 0, j \ne n\}$ for $\eta \in \D$.
\end{lemma}

\begin{proof}
\[
\begin{aligned}
\ip{f}{J_{\eta}} &= \iint_{\T^2} \frac{f(z,w)}{p(z,w)} \frac{\bar{z}^n
  p(z,\eta)}{1-\bar{w}\eta} d\sigma(w) d\sigma(z) \\
&= \int_{\T} \frac{f(z,\eta)}{p(z,\eta)} p(z,\eta) \bar{z}^n
d\sigma(z) \\
&= \sum_{k\geq 0} \hat{f}(n,k) \eta^k.
\end{aligned}
\]
\end{proof}

\begin{lemma} \label{lem:L}
In $L^2(1/|p|^2 d\sigma)$, for all $\eta \in \C$
\[
L_{\eta} \perp \vee\{z^jw^k: j\ne n, 0\leq k < m\}
\]
and for $f \in \vee\{z^jw^k: j\in \Z, 0\leq k < m\}$
\[
\ip{f}{L_{\eta}} = \sum_{k=0}^{m-1} \hat{f}(n,k) \eta^k.
\]
\end{lemma}

\begin{proof}
Note $L_\eta = J_{\eta} - H_{\eta}$.  As $f \perp H_{\eta}$ and
$\ip{f}{J_{\eta}} = \sum_{k=0}^{m-1} \hat{f}(n,k) \eta^k$, we see that
the desired formula holds for $\eta \in \D$.  Since both sides are
polynomials in $\eta$, the formula holds for all $\eta \in \C$.
\end{proof}

\begin{corollary} 
In $L^2(1/|p|^2 d\sigma)$
\[
\begin{aligned}
\vee\{L_{\eta}: \eta \in \D\} = & \mcP_{2n,m-1} \ominus(\mcP_{n-1,m-1}
\vee z^{n+1}\mcP_{n-1,m-1})\\ 
=& \vee\{z^j w^k: j\in \Z, 0\leq k<m\} \\ &\ominus \vee\{z^jw^k: j\ne n,
0\leq k < m\}.
\end{aligned}
\]
\end{corollary}

\begin{proof}
We have already shown the $L_{\eta}$'s are in the orthogonal
complements on the right.  On the other hand, if any $f$ (in either
orthogonal complement space) is orthogonal to $L_{\eta}$ for all
$\eta$, then $\hat{f}(n,k) = 0$ for $0\leq k<m$, implying $f=0$.
\end{proof}

The next Proposition (see also Corollary \ref{cor:ehe0}) shows that
certain orthogonal subspaces are mapped into each other by
multiplication by $z$.

\begin{prop} \label{Lprop}
In $L^2(1/|p|^2 d\sigma)$
\[
\begin{aligned}
&\mcP_{\infty, m-1} \ominus \mcP_{n-1,m-1}= \vee\{z^j L_{\eta}: j \geq 0, \eta \in \D\}\\
=& \vee\{z^j w^k: j\in \Z, 0\leq k<m\} \ominus \vee \{z^jw^k:  j<n,
0\leq k<m\}\\
\end{aligned}
\]
\[
\begin{aligned}
&\mcP_{\infty, m} \ominus \mcP_{n-1,m}= \vee\{z^j G_{\eta}: j \geq 0, \eta \in \D\}\\
=& \vee\{z^j w^k: j\in \Z, 0\leq k\leq m\} \ominus \vee \{z^jw^k:  j<n,
0\leq k\leq m\}.
\end{aligned}
\]



\end{prop}

\begin{proof} 
By the Corollary, $z^jL_{\eta}$ is in the orthogonal complement spaces
for all $j \geq 0, \eta \in \D$.  On the other hand, if anything in
these orthogonal complements is orthogonal to $z^jL_{\eta}$ for all
$j\geq 0, \eta \in \D$, then such an element will have no Fourier
support in the set $\{(j,k): j \geq n, 0\leq k < m\}$ and will be
orthogonal to itself.

We get similar decompositions when we use $G_{\eta}$ instead of
$L_{\eta}$ and allow $k=m$.
\end{proof}

If we apply the anti-unitary reflection operation $\rev{\cdot}$ at
degree $(n-1,m-1)$ we get other useful decompositions
\begin{multline}
\vee\{z^jw^k: j<n, 0\leq k <m \} \ominus \mcP_{n-1,m-1} =
\vee\{z^{j-n}L_{\eta}: j<0, \eta \in \D\} \\ \label{Lpropr}
=\vee\{z^jw^k: j\in \Z, 0\leq k < m\}\ominus \vee\{z^jw^k: j\geq 0,
0\leq k < m\}.
\end{multline}
Multiplication of the above equation by $z$ gives the following 
important consequence which provides
necessary conditions for the full measure characterization in Section
\ref{fullmeasure}.

\begin{corollary} \label{cor:neccond}
In $L^2(1/|p|^2 d\sigma)$,
\[
\mcE_{n,M}^2 \perp z\mcP_{\infty,M}
\]
for all $M\geq m-1$.
\end{corollary}

We get this for all $M\geq m-1$ simply because we can view $p$ as a
polynomial of degree at most $(n,M+1)$ for any $M\geq m-1$. 
Similarly, we obtain the following corollary.
\begin{corollary} \label{cor:ehe0}
In $L^2(1/|p|^2 d\sigma)$,
\[
\mcP_{\infty, M} \ominus \mcP_{n,M}=z(\mcP_{\infty, M} \ominus \mcP_{n-1,M}),
\]
for all $M\geq m-1$.
\end{corollary}

\section{Bernstein-Szeg\H{o} condition implies split-shift condition}\label{BSCiSSC}
Using the same setup as the previous section, we now delve into the
more refined orthogonalities necessary to prove that the
Bernstein-Szeg\H{o} condition implies the split-shift condition in
Theorem \ref{mainthm}.
Write $p(z,0) = a(z) b(z)$ where $a$ has no zeros in $\cD$ and $b$ has
all zeros in $\D$.  Let $\beta := \deg b$ and $\rev{b}(z) = z^\beta
\overline{b(1/\bar{z})}$.

\begin{lemma} \label{lem:Jperp}
In $L^2(1/|p|^2 d\sigma)$, for $\eta \in \D$, we have
\[
a z^j \perp w J_{\eta}
\]
for all $j< \beta$.

For $|\eta| > 1$,
\[
w^m \rev{b} z^j \perp H_{\eta}
\]
for all $j<n-\beta$.
\end{lemma}

\begin{proof}
Observe that for $j < \beta$ and $0<|\eta|<1$, $\ip{az^j}{wJ_{\eta}}$
equals
\[
\begin{aligned}
&\iint_{\T^2} \frac{a(z)z^j}{p(z,w)}
\frac{\bar{w}\bar{z}^n p(z,\eta)}{1-\bar{w}\eta} \frac{dw}{2\pi i w}
d\sigma(z) 
= \int_{\T}
z^{j-n} a(z)p(z,\eta) \int_{\T} \frac{dw}{2\pi i p(z,w)(w-\eta)w} d\sigma(z) \\ 
&\qquad= \int_{\T} z^{j-n} a(z) p(z,\eta)\left(\frac{1}{\eta p(z,\eta)} - \frac{1}{\eta
  p(z,0)} \right) d\sigma \\
&\qquad= \frac{1}{\eta} \left(\int_{\T}\bar{z}^{n-j} a(z) d\sigma(z) - \int_{\T}
\frac{z^{j-n}p(z,\eta)}{b(z)} d\sigma(z)\right) \\
&\qquad= -\frac{1}{\eta} \overline{\int_{\T}
  \frac{z^{n+\beta-j}\overline{p(z,\eta)} } {\rev{b}(z)} d\sigma(z)} = 0
\end{aligned}
\]
since $n-j > \deg a$.  The proof is easier when $\eta = 0$.

For $|\eta|>1$, $\overline{\ip{w^m\rev{b} z^j}{H_{\eta}}}$ equals
\[
\begin{aligned}
\ip{w^m\rev{b}z^j}{z^n\frac{\rev{p} \overline{\rev{p}(z,\eta)}}{1-w\bar{\eta}} }^* &= 
 \int_{\T} z^{n-\beta-j}b \bar{\eta}^m p(z,1/\bar{\eta}) \int_{\T}
\frac{1}{p(z,w)\bar{\eta}(1/\bar{\eta}-w)} \frac{dw}{2\pi i w} d\sigma(z) \\
&= \int_{\T} z^{n-\beta-j} b \bar{\eta}^{m} p(z,1/\bar{\eta})\left(
-\frac{1}{p(z,1/\bar{\eta})} +\frac{1}{p(z,0)} \right)
d\sigma(z) \\
&= -\int_{\T} z^{n-\beta-j}b \bar{\eta}^m d\sigma(z) + \int_{\T}
\bar{\eta}^m z^{n-\beta-j}
\frac{p(z,1/\bar{\eta})}{a(z)} d\sigma(z) \\
&= 0  
\end{aligned}
\]
for $n-\beta > j$.
\end{proof}

\begin{lemma} \label{lem:pperp}
In $L^2(1/|p|^2 d\sigma)$, if $f \in \vee \{z^j: j \geq 0\}$, then $f
\perp z^k p$ for all $k \geq 0$ if and only if $f(z) = a(z) q(z)$
where $q \in \C[z]$ has degree less than $\beta$.

If $f \in \vee\{w^m z^j: j \geq 0\}$, then $f \perp z^k \rev{p}$ for
all $k \geq 0$ if and only if $f(z,w) = w^m\rev{b}(z) q(z)$ where $q \in
\C[z]$ has degree less than $n-\beta$.
\end{lemma}

\begin{proof}
If $z^k p \perp f \in \vee\{z^j: j \geq 0\}$ for all $k\geq 0$ then
\[
 0 = \iint_{\T^2} \bar{z}^k \frac{f(z)}{p(z,w)} d\sigma(w)d\sigma(z) =
 \int_{\T} \frac{\bar{z}^k f(z)}{p(z,0)} d\sigma(z)
\]
for all $k \geq 0$ implies $f(z)/p(z,0) =\bar{z}
\overline{g(z)}$ for $g \in H^2(\T)=\vee \{z^j: j \geq 0\}$.  Then, $f(z) =
p(z,0)\bar{z}\overline{g(z)}$ and so 
\[
z^{n-1} \overline{f(z)} = z^n\overline{p(z,0)} g(z) \in H^2
\]
implies $f(z)$ is a polynomial of degree at most $n-1$.  In addition,
$\rev{f} = \rev{a} \rev{b} g$ implies $\rev{a}$ divides $\rev{f}$. (We
reflect $b$ at degree $\beta$ and $a$ at degree $n-\beta$.)  So,
$\rev{f} = \rev{a} h$ where $h \in \C[z]$ has degree less than
$\beta$.  Finally, $f = a q$ where $q \in \C[z]$ has degree less than
$\beta$.

For the converse, let $f = a q$.  Then,
\[
\begin{aligned}
&\iint_{\T^2} \frac{\bar{z}^k f(z)}{p(z,w)} d\sigma(w)d\sigma(z) =
\int_{\T} \frac{\bar{z}^kf(z)}{p(z,0)}d\sigma(z) 
= \int_{\T} \frac{\bar{z}^k q(z)}{b(z)} d\sigma(z) \\
&\qquad= \overline{\int_{\T}\frac{z^{k+\beta} \overline{q(z)}}{\rev{b}(z)}
    d\sigma(z)} 
= \overline{\int_{\T} \frac{z^{k+1} \rev{q}(z)}{\rev{b}(z)} d\sigma(z)} = 0
\end{aligned} 
\]
for all $k \geq 0$.

The proof of the second part is very similar.
\end{proof}

Define
\[
\mcK:= \mcP_{\infty,m} \ominus w\mcP_{\infty,m-1}
\]
\[
\mcL := \mcP_{\infty,m} \ominus \mcP_{\infty,m-1}.
\]

Notice
\[
\vee\{z^jp: j\geq 0\} \subset  \mcK
\]
and
\[
\vee\{z^j\rev{p}:j \geq 0\} \subset \mcL.
\]  

Let $P_0$ denote orthogonal projection onto $\mcP_{\infty,m-1}$ and
$P_0^{\perp} = I-P_0$.

Let $P_1$ denote orthogonal projection onto $w\mcP_{\infty,m-1}$ and
$P_1^{\perp} = I-P_1$.

The next two lemmas and corollary construct the spaces $\mcK_1$ and
$\mcK_2$ in the split-shift condition in Definition~\ref{SHC}.

\begin{lemma} \label{lem:KL}
In $L^2(1/|p|^2 d\sigma)$
\[
\mcK \ominus \vee\{z^jp: j \geq 0\} = P_1^{\perp}(\vee\{az^j: 0\leq j<\beta\})
\]
\[
\mcL \ominus \vee\{z^j \rev{p}: j \geq 0\} = P_0^{\perp}
(\vee\{w^m\rev{b}z^j: 0 \leq j< n-\beta\}).
\]

\end{lemma}

\begin{proof}
Let $f \in \mcK$ and $f \perp z^jp$ for all $j\geq 0$.  Write $f(z,w)
= f(z,0) - wg(z,w)$ where $g\in \mcP_{\infty,m-1}$ and notice that $P_1f = 0 =
P_1(f(z,0)) - wg(z,w)$ so that $f = f(z,0) - P_1f(z,0) =
P_1^{\perp}f(z,0)$.  Since $P_1f(z,0) \perp z^j p$ for all $j \geq 0$,
we see that $f(z,0) \perp z^j p$ for all $j \geq 0$.  By Lemma
\ref{lem:pperp}, $f(z,0) = a(z) q(z)$ where $\deg q < \beta$.  This
shows the inclusion $\subset$.

On the other hand, $P_1^{\perp}(a z^j) = a z^j - P_1 a z^j \in \mcK$,
$P_1(a z^j) \perp z^k p$ for all $k\geq 0$, and $a z^j \perp z^kp$ for
all $k\geq 0$ by Lemma \ref{lem:pperp}.

The second equation has a similar proof.

\end{proof}

\begin{lemma} \label{lem:PP}
In $L^2(1/|p|^2 d\sigma)$
\[
P_1^{\perp}(\vee\{az^j: 0\leq j<\beta\}) \subset \mcP_{n-1,m}
\]
\[
P_0^{\perp}(\vee\{w^m\rev{b} z^j: 0\leq j < n-\beta\}) \subset \mcP_{n-1,m}.
\]

\end{lemma}

\begin{proof}
For $0\leq j<\beta$, $P_1^{\perp}(a z^j) = a z^j - P_1(a z^j)$.  Clearly,
$a z^j \in \mcP_{n-1,m}$, so the main thing to show is that
$P_1(a z^j)\in \mcP_{n-1,m}$.

For $k\geq 0$, $\eta \in \D$
\[
\begin{aligned}
\ip{P_1(a z^j)}{wz^k L_{\eta}} &= \ip{a z^j}{wz^k L_{\eta}}\qquad  \text{
  since } wz^kL_{\eta} \in w\mcP_{\infty,m-1}\\
&= \ip{a z^{j-k}}{w L_{\eta}} \\
&=  \ip{a z^{j-k}}{wJ_{\eta}} \qquad \text{ since }a z^{j-k} \perp
 wH_{\eta} \text{ when } |\eta| <1 \\
&=0
\end{aligned}
\]
by Lemma \ref{lem:Jperp}. On the other hand, since $f =
\bar{w}P_1(a z^j)$ is an element of $\mcP_{\infty, m-1}$
\[
\ip{P_1(a z^{j})}{wz^k L_{\eta}} = \ip{f\bar{z}^k}{L_{\eta}} =
\sum_{t=0}^{m-1} \hat{f}(n+k,t) \eta^t \equiv 0
\]
by Lemma \ref{lem:L}.
So, $\hat{f}(j,k) = 0$ for $j\geq n$ and $k\in \Z$.  This shows
$P_1(a z^j) = wf \in \mcP_{n-1,m}$.

For the second inclusion, the main thing to show is $f = P_0
(w^m\rev{b}z^j) \in \mcP_{n-1,m}$ for $0\leq j < n-\beta$.  For $\eta \in
\C$,
\[
\ip{f}{z^k L_{\eta}} = \sum_{t=0}^{m-1} \hat{f}(n+k,t)\eta^{t}
\]
on one hand, while for $|\eta| >1$
\[
\begin{aligned}
\ip{f}{z^k L_{\eta}} &= \ip{w^m \rev{b}z^j}{z^kL_{\eta}} \\
&= \ip{w^m z^{j-k}\rev{b}}{-H_{\eta}} \text{ since } J_{\eta} \perp
w^m z^{j-k}\rev{b} \text{ when } |\eta| >1 \\
&= 0 
\end{aligned}
\]
by Lemma \ref{lem:Jperp} since $j-k < n-\beta$.
This implies $\hat{f}(n+k,t) = 0$ for $k\geq 0$ and $0\leq t< m$ as
desired.

\end{proof}

Set
\[
\mcK_1 = \mcK \ominus \vee\{z^jp: j \geq 0\}.
\]
\[
\mcL_1 = \mcL \ominus \vee\{z^j \rev{p}: j \geq 0\}.
\]


\begin{corollary} \label{cor:Ksubset}
In $L^2(1/|p|^2 d\sigma)$
\[
\begin{aligned}
\mcK_1 &= P_{\mcE_{n-1,m}^1} (\vee\{a z^j: 0\leq j < \beta\}) \subset
\mcE_{n,m}^1,\\
\mcL_1 &= P_{\mcF_{n-1,m}^1} (\vee\{w^m\rev{b}z^j: 0\leq j < n-\beta\})
\subset \mcF_{n,m}^1,\\
\rev{\mcL}_1 &= P_{\mcE_{n-1,m}^1} (\vee\{bz^j: 0\leq j < n-\beta\}) \subset
\bar{z} \mcE_{n,m}^1.
\end{aligned}
\]
\end{corollary}

\begin{proof}
By Lemmas \ref{lem:KL} and \ref{lem:PP}, $\mcK_1$ is contained in both
$\mcE_{n-1,m}^1$ and $\mcE_{n,m}^{1}$.  Let $P_{n-1,m-1}^1$ denote
orthogonal projection onto $w\mcP_{n-1,m-1}$.  For any $f = az^j -
P_1(az^j)$ we know $f \in \mcP_{n-1,m}$ for $0\leq j< \beta$, and so
we see that $P_1(az^j) = P_{n-1,m-1}^1 P_1(az^j) =
P_{n-1,m-1}^1(az^j)$.  Therefore, $f = az^j - P_{n-1,m-1}^1(az^j) =
P_{\mcE_{n-1,m}^1}(az^j)$, which proves $\mcK_1 = P_{\mcE_{n-1,m}^1}
(\vee\{a z^j: 0\leq j < \beta\})$.

The second set of equations has a similar proof.  The last set of
equations follows from the second set by taking the reflection
operation $\rev{\cdot}$ at the degree $(n-1,m)$.
\end{proof}

Let $\mcK_2 := \mcE_{n-1,m}^1 \ominus \mcK_1$
so that 
\[
\mcE_{n-1,m}^1 = \mcK_1 \oplus \mcK_2.
\]

Similarly, define $\mcL_2$ so that
\[
\mcF_{n-1,m}^1 = \mcL_1\oplus \mcL_2.
\]

The next lemma gives a different characterization of the spaces
$\mcK_2$ and $\mcL_2$.

\begin{lemma}
In $L^2(1/|p|^2 d\sigma)$
\begin{multline}
\vee\{z^jp:j \geq 0\} \oplus \vee\{wz^j L_{\eta}: j\geq 0, \eta \in
\D\}\\ \label{eq:pforward}
= \mcK_2 \oplus \vee\{z^j G_{\eta}: j\geq 0, \eta \in
\D\}
\end{multline}
\begin{multline}
\vee\{z^j\rev{p}:j \geq 0\} \oplus \vee\{z^j L_{\eta}: j\geq 0, \eta \in
\D\}\\ \label{eq:revpback}
= \mcL_2 \oplus \vee\{z^j G_{\eta}: j\geq 0, \eta \in
\D\}.
\end{multline}
\end{lemma}

\begin{proof}
Now, 
\[
\begin{aligned}
&\mcP_{\infty, m} \ominus w \mcP_{n-1,m-1} = (\mcP_{\infty,m} \ominus
  w\mcP_{\infty,m-1}) \oplus w(\mcP_{\infty,m-1} \ominus \mcP_{n-1,m-1})\\
=& \mcK \oplus \vee\{wz^j L_{\eta}: j\geq 0, \eta \in \D\} \\
=& \mcK_1 \oplus \vee\{z^jp:j \geq 0\} \oplus \vee\{wz^j L_{\eta}: j\geq 0, \eta \in \D\}
\end{aligned}
\]
by Proposition \ref{Lprop}.  The same set is equal to 
\[
\mcE_{n-1,m}^1 \oplus \vee\{z^j G_{\eta}: j\geq 0, \eta \in \D\} =
\mcK_1 \oplus \mcK_2 \oplus \vee\{z^j G_{\eta}: j\geq 0, \eta \in \D\}
\]
and after canceling $\mcK_1$ we obtain \eqref{eq:pforward}.
The proof for $\mcL_2$ follows along the same lines by considering $\mcP_{\infty, m} \ominus \mcP_{n-1,m-1}$. 
\end{proof}

\begin{lemma}
In $L^2(1/|p|^2 d\sigma)$
\begin{multline} \label{eq:pback}
\vee\{z^j p: j<0\} \oplus \vee\{wz^{j-n} L_{\eta}: j<0, \eta \in \D\} 
\\
=\rev{\mcL}_2 \oplus \vee\{z^{j-n} G_{\eta}: j< 0, \eta \in
\D\}
\end{multline}
where $\rev{\mcL}_2$ is obtained by reflecting $\mcL_2$ at degree
$(n-1,m)$ and 
\begin{multline} \label{eq:backandforth}
\vee\{z^jp: j \in \Z\} \oplus \vee\{wz^{j-n} L_{\eta}: j<0, \eta \in \D\}
\oplus \vee\{wz^j L_{\eta}: j\geq 0, \eta \in
\D\} \\
= \vee\{z^{j-n} G_{\eta}: j< 0, \eta \in
\D\} \oplus \vee\{z^{j} G_{\eta}: j\geq 0, \eta \in
\D\} \oplus \mcE_{n-1,m}^1.
\end{multline}
\end{lemma}

\begin{proof}
The first part follows from applying the reverse operation
$\rev{\cdot}$ at the degree $(n-1,m)$ in \eqref{eq:revpback} and by using \eqref{Lsymm} and \eqref{Gsymm}. 

The second part comes from decomposing 
\[
\vee\{z^jw^k: j \in \Z, 0\leq k \leq m\} \ominus w\mcP_{n-1,m-1} 
\]
in two different ways.  By Proposition \ref{Lprop}, equation \eqref{Lpropr} and Lemma \ref{lem:p} it equals
\[
\begin{aligned}
&\vee\{z^jw^k: j \in \Z, 0\leq k \leq m\} \ominus \vee\{z^jw^k: j\in \Z,
  1\leq k\leq m\} \\
 &\oplus \vee\{wz^{n-j}L_{\eta}: j<0,\eta \in \D\}
  \oplus \vee\{wz^jL_{\eta}: j\geq 0, \eta \in \D\} \\
 =& \vee\{z^jp: j
  \in \Z\} \oplus \vee\{wz^{j-n} L_{\eta}: j<0, \eta \in \D\} \oplus
  \vee\{wz^j L_{\eta}: j\geq 0, \eta \in \D\},
\end{aligned}
\]
while it also equals the right hand side of \eqref{eq:backandforth}.
\end{proof}

\begin{theorem} \label{BStoSS}
Assume $p \in \C[z,w]$ has no zeros in $\T\times\cD$ and degree at
most $(n,m)$.  In $L^2(1/|p|^2 d\sigma)$, for
\[
\mcK_1 = P_{\mcE_{n-1,m}^1} (\vee\{a z^j: 0\leq j < \beta\})
\]
\[
\rev{\mcL}_1 = P_{\mcE_{n-1,m}^1} (\vee\{bz^j: 0\leq j < n-\beta\})
\]
we have
\[
\mcE_{n-1,m}^1 = \mcK_1 \oplus \rev{\mcL}_1
\]
and
\[
\mcE_{n,m}^1 = \mcK_1 \oplus z\rev{\mcL}_1 \oplus \C p.
\]
Consequently, the split-shift orthogonality condition holds for a
positive linear form associated to a Bernstein-Szeg\H{o} measure with
$p$ having no zeros in $\T\times \cD$.

\end{theorem}
 
\begin{proof} 
By Corollary \ref{cor:Ksubset}, it is enough to prove
\[
\rev{\mcL}_1 = \mcK_2
\]
and
\[
\mcE_{n,m}^1 = \mcK_1 \oplus z\mcK_2 \oplus \C p.
\]

The direct sum of the left sides of \eqref{eq:pforward} and
\eqref{eq:pback} yields the left side of \eqref{eq:backandforth}.  So,
the direct sum of the corresponding right hand sides are equal which
means
\[
\begin{aligned}
&\mcK_2 \oplus \vee\{z^j G_{\eta}: j\geq 0, \eta \in
\D\} \oplus \rev{\mcL}_2 \oplus \vee\{z^{j-n} G_{\eta}: j< 0, \eta \in
\D\} \\
=& \vee\{z^{j-n} G_{\eta}: j< 0, \eta \in
\D\} \oplus \vee\{z^{j} G_{\eta}: j\geq 0, \eta \in
\D\} \oplus \mcE_{n-1,m}^1.
\end{aligned}
\]
Therefore, $\mcE_{n-1,m}^1 = \mcK_2 \oplus \rev{\mcL}_{2}$. But,
$\mcE_{n-1,m}^1 = \rev{\mcF}_{n-1,m}^{1} = \rev{\mcL}_{1} \oplus
\rev{\mcL}_{2}$ and so $\mcK_2 = \rev{\mcL}_1$.

We know $\C p, \mcK_1 \subset \mcE_{n,m}^1$ by Lemma \ref{lem:p} and
Corollary \ref{cor:Ksubset}.  By definition of $\mcK_1$ we know
$\mcK_1 \perp p$.  

Now, using Corollary \ref{cor:ehe0}, we see that 
\[
\vee\{z^jw^k: j\geq 0, 0\leq k \leq m\} \ominus \vee\{z^jw^k: 0\leq j
\leq n, 1\leq k \leq m\}
\]
decomposes into
\[
\begin{aligned}
&\mcK \oplus zw(\mcP_{\infty,m-1} \ominus \mcP_{n-1,m-1})\\
 =&
\vee\{z^jp: j \geq 0\} \oplus \mcK_1 \oplus \vee\{zwz^j L_{\eta}: j
\geq 0, \eta \in \D\}\\
=& \C p \oplus \mcK_1 \oplus z(\vee\{z^jp:j\geq 0\}) \oplus z(\vee\{wz^j L_{\eta}: j
\geq 0, \eta \in \D\}) \\
=& \C p \oplus \mcK_1 \oplus z(\mcK_2 \oplus \vee\{z^jG_{\eta}:j\geq
0, \eta \in \D\}) \text{ by \eqref{eq:pforward} } \\
=&  \C p \oplus \mcK_1 \oplus z\mcK_2 \oplus z(\vee\{z^jG_{\eta}:j\geq
0, \eta \in \D\})
\end{aligned}
\]
but it also decomposes into
\[
\mcE_{n,m}^1 \oplus z(\vee\{z^jG_{\eta}:j\geq
0, \eta \in \D\})
\]
and therefore
\[
\mcE_{n,m}^1 = \C p \oplus \mcK_1 \oplus z\mcK_2.
\]
\end{proof}

\section{Split-shift condition implies Bernstein-Szeg\H{o} condition} \label{sec:SStoBS}

The goal now is to prove that the split-shift condition (see Definition \ref{SHC}) implies that
$\mcT$ can be represented using a Bernstein-Szeg\H{o} measure whose
associated polynomial has no zeros on $\T\times \cD$.

We call the pair $(\mcK_1,\mcK_2)$ a \emph{shift-split} of
$\mcE_{n,m}^1$.  By dimensional considerations $\mcE_{n,m}^1
\ominus(\mcK_1\oplus z\mcK_2)$ will be one dimensional, and therefore
of the form $\C p$ for some unit norm $p$.  We shall call $p$ a
\emph{split-poly} associated to the shift-split.  The point now will be
to prove that a split-poly $p$ has no zeros on $\T \times \cD$ and along
the way we will prove some interesting formulas for $p$ (which will
also give formulas for an arbitrary $p$ with no zeros on $\T\times
\cD$ since we can apply our formulas to $1/|p|^2 d\sigma$).

There may be more than one shift-split of $\mcE_{n,m}^1$, but we shall
see that each split-poly is associated to one shift-split.  We will
provide a description of all split-polys (and hence all shift-splits via
the previous section) in Section \ref{shiftsplits}.
  
 Let $K_{j,k}$ be the reproducing kernel for $\mcP_{j,k}$ in
 $\mcH_{\mcT}$. Namely, for $(\zeta,\eta) \in \C^2$,
 $(K_{j,k})_{(\zeta,\eta)}(\cdot,\cdot) =
 K_{j,k}(\cdot,\cdot;\zeta,\eta)$ is the unique element of
 $\mcP_{j,k}$ such that
\[
\ip{f}{(K_{j,k})_{(\zeta,\eta)}}_{\mcT} = f(\zeta,\eta)
\]
for all $f \in \mcP_{j,k}$.

\begin{remark} \label{RKremark}
We shall use some standard facts about reproducing kernels of
polynomials on $\T^2$. See Section 3 of \cite{gK08}.

\begin{enumerate}
\item The reproducing kernel of an orthogonal direct sum is the sum of
  the reproducing kernels. 

\item Shifting a subspace by $z$ (resp. $w$) multiplies the
  reproducing kernel by $z\bar{\zeta}$ (resp. $w\bar{\eta}$).  

\item The ``reflection'' $\rev{\cdot}$ of a subspace ``reflects'' the
  reproducing kernel.
\end{enumerate}

On this last point, if $\mcH$ is a subspace of polynomials of degree
at most $(j,k)$ and $H$ is its reproducing kernel, the subspace
\[
\rev{\mcH} := \{ z^jw^k\bar{f}(1/z,1/w): f \in \mcH\}
\]
has reproducing kernel
\[
\rev{H}(z,w;\zeta,\eta) := (z\bar{\zeta})^j(w\bar{\eta})^k
H(1/\bar{\zeta},1/\bar{\eta}; 1/\bar{z}, 1/\bar{w}).
\]
The degree $(j,k)$ at which we reflect will either be mentioned
explicitly or will be the maximal degree of the elements of the
subspace.

\end{remark}

Using these manipulations we get the following formulas.
 \[
 \begin{aligned}
 E_{j}^1 &= K_{j,m} - w\bar{\eta} K_{j,m-1} = \text{ the reproducing kernel
 for } \mcE_{j,m}^1\\
 F_{j}^1 &= \rev{E}_{j}^1 = K_{j,m} - K_{j,m-1} = \text{ the reproducing
 kernel for } \mcF_{j,m}^1\\
 E_{k}^2 &= K_{n,k} - z\bar{\zeta} K_{n-1,k} = \text{ the
 reproducing kernel for } \mcE_{n,k}^2\\
 F_{k}^2 &= \rev{E}_{k}^2 = K_{n,k} - K_{n-1,k} = \text{ the
 reproducing kernel for } \mcF_{n,k}^2.
 \end{aligned} 
 \]
 For example, the first formula follows from the orthogonal
 decomposition
\[
\mcP_{j,m} = w\mcP_{j,m-1} \oplus \mcE_{j,m}^1.
\]

 We record some basic formulas which do not require any special
 orthogonality conditions.  In fact, they are just the result of
 manipulating the equations above.

 \begin{lemma} \label{lem:subtract} We have
 \[
 E_{j}^1(z,w;\zeta,\eta) - F_{j}^1(z,w;\zeta,\eta) =
 (1-w\bar{\eta})K_{j,m-1}(z,w;\zeta,\eta)
 \]
 \[
 E_{k}^2(z,w;\zeta,\eta) - F_{k}^2(z,w;\zeta,\eta) =
 (1-z\bar{\zeta})K_{n-1,k}(z,w;\zeta,\eta).
 \]
  \end{lemma}

 If $\{E_0(z,w),\dots,E_{m}(z,w)\}$ is an orthonormal basis for
 $\mcE_{n,m}^2$ then we write
 \[
 E^2_m(z,w) =
 (E_0(z,w), \dots, E_{m}(z,w)) = (1,w,\dots, w^m) E_m^2(z)
 \]
 for an appropriate $(m+1)\times (m+1)$ matrix polynomial $E^2_m(z)$.  
 Then,
 \[
 E_{m}^2(z,w;\zeta,\eta) = E_{m}^2(z,w) E_m^2(\zeta,\eta)^* =
 (1,w,\dots, w^m) E_m^2(z) E_m^2(\zeta)^* (1,\eta,\dots, \eta^m)^*.
 \]


 \begin{lemma} \label{Einv}
 The matrix polynomial $E^2_m(z)$  is invertible for all $z \in \cD$.
 \end{lemma}

 \begin{proof} 
 Suppose $E^2_m(z_0)$ is singular for some $z_0 \in \C$ and choose
 nonzero $v \in \C^{m+1}$ such that $E^2_m(z_0)v = 0$.  Then,
 \[
 f(z,w) = E^2_m(z,w) v = (1,w,\dots, w^m) E_m^2(z) v
 \]
 is in $\mcE_{n,m}^2$, and $f(z_0,w) = 0$ for all $w$.  So, $f(z,w) =
 (z-z_0)g(z,w)$ for some $g \in \mcP_{n-1,m}$.  Since $f \perp z g$ we
 have
 \[
 \|f-zg\|^2 = \|f\|^2 + \|g\|^2 = |z_0|^2 \|g\|^2.
 \]
 Then, $\|f\|^2 = \|v\|^2 = (|z_0|^2-1)\|g\|^2$ which implies $|z_0| >
 1$.
 \end{proof}

 Let $K_1$ be the reproducing kernel for $\mcK_1$, and let $K_2$ be
 the reproducing kernel for $\mcK_2$.  When $p$ has unit norm, the
 reproducing kernel for $\C p$ is $p(z,w)\overline{p(\zeta,\eta)}$ but
 we will simply write $p\bar{p}$.

 \begin{lemma} \label{lem:OC} If $(K_1,K_2)$ is a shift-split of
   $\mcE_{n,m}^1$ with split-poly $p$ then
 \[
 E_{n-1}^1 = K_1 + K_2
 \]
 \[
 E_{n}^1 = K_1 + z\bar{\zeta} K_2 + p\bar{p}
 \]
 \[
 F_{n-1}^1 = \rev{K}_1 + \rev{K}_2
 \]
 \[
 F_{n}^1 = z\bar{\zeta} \rev{K}_1 + \rev{K}_2 + \rev{p}\bar{\rev{p}}
 \]
where the kernels $\rev{K}_1$ and $\rev{K}_2$ are reflected at the
degree $(n-1,m)$.
 \end{lemma}

\begin{proof}
 These all follow from Remark \ref{RKremark} and the definition of
 shift-split.
\end{proof}

 \begin{theorem} \label{aglerdecomps}
 If $(K_1,K_2)$ is a shift-split of $\mcE_{n,m}^1$ with split-poly $p$
 then
 \[
 \begin{aligned}
 p\bar{p} - \rev{p}\bar{\rev{p}} &=
 (1-w\bar{\eta})E_{m-1}^2 + (1-z\bar{\zeta})(\rev{K}_2 - K_1) \\
 &= (1-w\bar{\eta}) F_{m-1}^2 + (1-z\bar{\zeta})(K_2 - \rev{K}_1) \\
 &= (1-w\bar{\eta}) F_{m-1}^2 + (1-z\bar{\zeta})(\rev{K}_2 - K_1) +
 (1-z\bar{\zeta})(1-w\bar{\eta})K_{n-1,m-1}
 \end{aligned}
 \]
and
 \[
 \begin{aligned}
 p\bar{p} - w\bar{\eta} \rev{p}\bar{\rev{p}} &=
 (1-w\bar{\eta})E_{m}^2 + (1-z\bar{\zeta})(w\bar{\eta}\rev{K}_2 - K_1) \\
 &= (1-w\bar{\eta}) F_{m}^2 + (1-z\bar{\zeta})(K_2 - w\bar{\eta}\rev{K}_1) \\
 &= (1-w\bar{\eta}) F_{m}^2 + (1-z\bar{\zeta})(w\bar{\eta}\rev{K}_2 - K_1) +
 (1-z\bar{\zeta})(1-w\bar{\eta})K_{n-1,m}.
 \end{aligned}
 \]

 \end{theorem}

 \begin{proof}
 Combining Lemmas \ref{lem:subtract} and \ref{lem:OC}, we get
 \[
 z\bar{\zeta}(K_1+K_2 - (\rev{K}_1 + \rev{K}_2)) =
 (1-w\bar{\eta})z\bar{\zeta} K_{n-1,m-1}
 \]
 and
 \[
 K_1 + z\bar{\zeta}K_2 + p\bar{p} -(z\bar{\zeta}\rev{K}_1 + \rev{K}_2 +
 \rev{p}\bar{\rev{p}})
 = (1-w\bar{\eta})K_{n,m-1}.
 \]
 Subtract these two formulas to get
 \[
 (1-z\bar{\zeta})(K_1-\rev{K}_2) +p\bar{p} - \rev{p}\bar{\rev{p}} =
 (1-w\bar{\eta}) E_{m-1}^2
 \]
 which rearranges to get the first desired formula.  Similar arguments
 give the remaining formulas.
 \end{proof}

If $\mcT$ comes from a Bernstein-Szeg\H{o} measure $1/|p|^2 d\sigma$
where $p$ has no zeros on $\T\times \cD$, then Theorem \ref{BStoSS}
implies $p$ is a split-poly and then the above theorem immediately
implies Theorem \ref{sosthm}, the sum of squares theorem from the
introduction since reproducing kernels can be written as a sum of
squares of an orthonormal basis.  In general, the sum of squares
formula implies a split-poly has no zeros on $\T\times \cD$.

 \begin{corollary} If $p$ is a split-poly, then $p$ has no
   zeros on $\T\times \cD$.

 \end{corollary}

 \begin{proof}
 We use the second set of formulas in Theorem \ref{aglerdecomps}.
 Suppose $p(z,w) = 0$ for some $(z,w) \in \T\times \cD$.  Setting $z =
 \zeta \in \T$ and $w=\eta \in \cD$ we have
 \[
 |p(z,w)|^2 - |w|^2|\rev{p}(z,w)|^2 = - |w\rev{p}(z,w)|^2 = 
 (1-|w|^2) E_{m}^{2}(z,w;z,w) \geq 0
 \]
 which shows $w\rev{p}(z,w) = 0$.  Then, for arbitrary $\eta \in \C$ we
 have
 \[
 0 = E_{m}^2(z,w;z,\eta) = E_{m}^2(z,w) E_{m}^2(z,\eta)^* = (1,w,\dots,
 w^m)E_{m}^2(z) E_{m}^2(z)^* (1,\eta,\dots,
 \eta^m)^*
 \]
 which implies
 \[
 0 = (1,w,\dots,w^m)E_{m}^2(z) E_{m}^2(z)^*
 \]
 contradicting the fact that $E_m^2(z)$ is invertible from Lemma
 \ref{Einv}.
 \end{proof}

We can now prove the split-shift orthogonality condition implies the
Bernstein-Szeg\H{o} condition in Theorem \ref{mainthm}.

 \begin{corollary} \label{SStoBS}
 Suppose two positive linear forms $\mcT_1$ and $\mcT_2$ both satisfy
 the split-shift condition with the same split-poly $p$.  Then, $\mcT_1
 = \mcT_2$ and the linear forms agree with the linear form associated
 with the measure $1/|p|^2d\sigma$.
 \end{corollary}

 \begin{proof} It is enough to show the reproducing kernels $K_{n,m}$
   are the same for both forms.  We can form a matrix polynomial
   $E_m^2(z)$ corresponding to each form $\mcT_1$ and $\mcT_2$, say
   $E_1(z)$ and $E_2(z)$ (just in this proof; we will not use this
   notation elsewhere).  Using Theorem \ref{aglerdecomps} for $z=\zeta
   \in \T$ and arbitrary $w,\eta \in \C$ we get $E_1(z)E_1(z)^* =
   E_2(z)E_2(z)^*$ for $z \in \T$ using arguments similar to the
   previous proof.  Then,
$$E_2^{-1}(z) E_1(z)=\bar{E}_2(1/z)^t \bar{E}_1^{-1}(1/z)^t.$$ By
Lemma \ref{Einv}, the left hand side is analytic for $|z|\leq 1$,
while the right hand side is analytic for $|z|\geq 1$. By Liouville's
theorem $E_2^{-1}(z) E_1(z)=V$ is a constant unitary matrix.  This in
turn implies the reproducing kernels $E_{m}^2(z,w;\zeta,\eta)$ for
$\mcT_1$, $\mcT_2$ are the same.  By the next lemma, we may conclude
that $\mcT_1 = \mcT_2$.
 \end{proof}

\begin{lemma} The inner product in $\mcH_{\mcT}$ is determined by the
  reproducing kernel $E_{m}^2$.
\end{lemma}
\begin{proof}
Notice $F_{m}^2(z,w;\zeta,\eta) = (z\bar{\zeta})^n(w\bar{\eta})^m
E_{m}^2(1/\bar{\zeta}, 1/\bar{\eta}; 1/\bar{z}, 1/\bar{w})$.  So,
$E_{m}^2$ determines $F_{m}^2$.  By Lemma \ref{lem:subtract},
$E_{m}^2$ determines $K_{n-1,m}$ and since $K_{n,m} = E_{m}^2 +
z\bar{\zeta} K_{n-1,m}$ we see that $E_{m}^2$ determines $K_{n,m}$ as
well.
\end{proof}

We can now prove Theorem \ref{frthm}.

\begin{proof}[Proof of Theorem \ref{frthm}]
We already know that if $t=|p|^2$ then the split-shift condition
holds.  On the other hand, if the split-shift condition holds with
split-poly $p$, then $p$ has no zeros on $\T\times \cD$ and the form
corresponding to $1/|p|^2d\sigma$ agrees with the form $\mcT$.
Then, by Cauchy-Schwarz
\[
1 = \left(\int_{\T^2} \frac{\sqrt{t}}{|p|} \frac{|p|}{\sqrt{t}} d\sigma\right)^2
\leq \int_{\T^2} \frac{t}{|p|^2} d\sigma \int_{\T^2} \frac{|p|^2}{t}
d\sigma
= \int_{\T^2} \frac{t}{t}d\sigma \int_{\T^2} \frac{|p|^2}{|p|^2}
d\sigma = 1
\]
since the forms agree. Since we have equality in our application of
Cauchy-Schwarz, it is not hard to see $t = |p|^2$.  
\end{proof}

\section{The matrix condition}\label{se:MC}

The abstract flavor of the split-shift orthogonality condition makes
it difficult to check.  This section is devoted to showing it is
equivalent to checking that a number of natural operators vanish.

The ``matrix condition'' \eqref{mc} from Theorem \ref{mainthm} can be viewed as saying the smallest invariant subspace
of $T$ containing the range of $B$ is contained in the kernel of $A$.

\begin{prop} \label{sstomc}
If a positive linear form $\mcT$ satisfies the split-shift condition
with shift-split $(\mcK_1,\mcK_2)$, then the matrix condition
\eqref{mc} holds and
\begin{align}
\vee\{ (T^*)^jA^*f: f \in w \mcE_{n,m-1}^2, j=0,1,\dots\} &\subset
\mcK_1 \label{K1inclusion} \\
\vee\{ T^jBf: f \in w \mcF_{n,m-1}^2, j=0,1,\dots\} &\subset
\mcK_2. \label{K2inclusion} 
\end{align}
\end{prop}

\begin{proof}
Let $(\mcK_1, \mcK_2)$ be a shift-split of $\mcE_{n,m}^1$.  Then,
$\mcE_{n-1,m}^1 = \mcK_1 \oplus \mcK_2$ and $\mcE_{n,m}^1 = \mcK_1
\oplus z\mcK_2 \oplus \C p$ where $p$ is the associated split-poly.

The strategy is to prove (1) the range of $B$ is contained in
$\mcK_2$, (2) $\mcK_2$ is an invariant subspace of $T$ ($T \mcK_2
\subset \mcK_2$), and (3) $A \mcK_2 = 0$.  This will imply that
\eqref{mc} holds as well as \eqref{K2inclusion}.

Since $\mcK_1 \subset \mcE_{n,m}^1$,  $\mcK_1 \perp w
\mcF_{n,m-1}^2$ and therefore for $g \in \mcK_1, f \in w \mcF_{n,m-1}^2$ 
\[
\ip{Bf}{g} = \ip{P_{\mcE_{n-1,m}^1} P_{w \mcF_{n,m-1}^2}f}{g} =
\ip{f}{ P_{w\mcF_{n,m-1}^2} g} = 0.
\]
So, the range of $B$ is orthogonal to $\mcK_1$ and therefore must be
contained in $\mcK_2$.

To show $T \mcK_2 \subset \mcK_2$, let $f_2 \in \mcK_2$ and $g_1 \in
\mcK_1$.  Since $z\mcK_2\perp \mcK_1$, we know $g_1 \perp z f_2$ and
therefore
\[
\ip{Tf_2}{g_1} = \ip{M_z f_2}{g_1} = 0.  \]
So, $T \mcK_2 \perp \mcK_1$ and thus $T\mcK_2
\subset \mcK_2$.

Finally, since $z\mcK_2 \subset \mcE_{n,m}^1 \perp w\mcE_{n,m-1}^2$,
we must have 
\[
A\mcK_2=  P_{w\mcE_{n,m-1}^2} M_z \mcK_2 = 0.
\]
The proof of \eqref{K1inclusion} is similar if we work with adjoints
of our operators.
\end{proof}

\begin{prop}\label{mctoss}
 Suppose $\mcT$ is a positive linear form on $\mcL_{n,m}$ satisfying
 the matrix condition \eqref{mc}.  Set
\[
\mcK_2 = \vee\{ T^jBf: f \in w \mcF_{n,m-1}^2, j=0,1,\dots\}
\]
and $\mcK_1 = \mcE_{n-1,m}^1 \ominus \mcK_2$. Then, $(\mcK_1, \mcK_2)$
is a shift-split of $\mcE_{n-1,m}^1$ and hence $\mcT$ satisfies the
split-shift orthogonality condition.
\end{prop}

\begin{proof}
Notice that by the Cayley-Hamilton theorem we do not need to consider
all powers of $T$ in the definition of $\mcK_2$, so that
\[
\mcK_2 = \vee\{ T^jBf: j=0,1,\dots,n-1, f \in w \mcF_{n,m-1}^2\}
\]
and also $AT^j B = 0$ for $j=0,1,2,\dots$.  

We need to show $z\mcK_2 \perp \mcK_1$ and $z\mcK_2, \mcK_1 \subset
\mcE_{n,m}^1$.  

To prove $z\mcK_2 \perp \mcK_1$, simply note that for $f \in w
\mcF_{n,m-1}^2$, $g \in \mcK_1$, $j=0,1,2,\dots$, we have
\[
\ip{zT^jBf}{g} = \ip{T^{j+1}Bf}{g} = 0.
\]
The proves $z\mcK_2\perp \mcK_1$ since $z\mcK_2$ is spanned by
elements of the form $zT^j Bf$.

To show $z\mcK_2 \subset \mcE_{n,m}^1$, note that 
\[
z\mcK_2 \subset \mcP_{n,m} \ominus zw\mcP_{n-1,m-1} = \mcE_{n,m}^1
\oplus w\mcE_{n,m-1}^2
\]
and therefore it is enough to show $z\mcK_2 \perp w \mcE_{n,m-1}^2$.
So, for $f\in w\mcF_{n,m-1}^2$ and $g \in w\mcE_{n,m-1}^2$ we have
\[
\ip{z T^jBf}{g} = \ip{AT^j Bf}{g} = 0 \quad j=0,1,2\dots
\]
since $A = P_{w\mcE_{n,m-1}^2} M_z$ and $AT^jB=0$.  This proves
$z\mcK_2 \subset \mcE_{n,m}^1$.

Similarly, to show $\mcK_1 \subset \mcE_{n,m}^1$ it is enough to show
$\mcK_1 \perp w\mcF_{n,m-1}^2$, since $\mcK_1 \subset \mcP_{n,m}
\ominus w\mcP_{n-1,m-1}$.  Observe that for $f \in w \mcF_{n,m-1}^2$
and $g \in \mcK_1$, $Bf \in \mcK_2$ and so
\[
0 = \ip{Bf}{g} = \ip{f}{g}
\]
and therefore $w\mcF_{n,m-1}^2 \perp \mcK_1$.
\end{proof}

\section{Description of shift-splits and
  split-polys} \label{shiftsplits}

 If the split-shift condition holds for $\mcT$, then we have seen that
 $\mcT$ can be represented using moments of a measure $1/|p|^2
 d\sigma$ where $p \in \C[z,w]$ has no zeros in $\T \times \cD$ and
 $p$ has degree at most $(n,m)$.  The description of all such $p$ is
 essentially an algebra problem.

\begin{lemma} 
Let $t(z,w)$ be a two variable trigonometric polynomial which is
factorable as $|p(z,w)|^2$ where $p$ has degree at most $(n,m)$ and no
zeros in $\T \times \cD$.  

Then, there exists a $g \in \C[z,w]$ with no zeros in $\T\times \cD$
none of whose irreducible factors involve $z$ alone, and there exists
a stable polynomial $q \in \C[z]$ (no zeros on $\cD$) such that
\[
t(z,w) = |q(z)g(z,w)|^2 \text{ for } (z,w) \in \T^2
\]
Moreover, if $t(z,w) = |p_1(z,w)|^2$ where $p_1$ has degree at most
$(n,m)$ and no zeros on $\T\times \cD$, then there exist $q_1, q_2 \in
\C[z]$ such that $q = q_1 \rev{q}_2$ and
\[
p_1(z,w) = q_1(z)q_2(z) g(z,w)
\]
 \end{lemma}

\begin{proof}  
Suppose $t=|p|^2$ as above.  We may factor $p(z,w) = h(z) g(z,w)$
where $g$ has no irreducible factors involving $z$ alone.  By the one
variable Fej\'er-Riesz lemma we can factor $|h|^2=|q|^2$ with $q$, a
stable one variable polynomial.  Then, $t(z,w) = |q(z)g(z,w)|^2$ on
$\T^2$. 

Now, if $t=|p_1|^2$ as above, then again $p_1(z,w) = h_1(z) g_1(z,w)$
where $g_1$ has no irreducible factors involving $z$ alone.  Now,
\[
h_1(z) g_1(z,w) \overline{h_1(z) g_1(z,w)} = h(z) g(z,w)
\overline{h(z) g(z,w)}
\]
on $\T^2$ which implies
\[
h_1(z) g_1(z,w) \rev{h}_1(z) \rev{g}_1(z,w) = q(z) g(z,w) \rev{q}(z) \rev{g}(z,w)
\]
on all of $\C^2$, when we reflect at appropriate degrees.
Then, for $z \in \T$
\[
\frac{h_1(z) g_1(z,w)}{q(z) g(z,w)} = \frac{\rev{q}(z)
  \rev{g}(z,w)}{\rev{h}_1(z) \rev{g}_1(z,w)}
\]
and the left side is holomorphic for all $w\in \cD$ and the right
side is holomorphic for $|w|\geq 1$ making the function entire and
rational in $w$. The same can be said for the reciprocal and this
forces the function to be constant in $w$.  So, for $z\in \T$
\[
\frac{h_1(z) g_1(z,w)}{q(z) g(z,w)} = \frac{h_1(z) g_1(z,0)}{q(z) g(z,0)}
\]
 and we see 
\[
g_1(z,w)g(z,0) = g(z,w)g_1(z,0).
\]
This extends to all $z \in \C$ and since $g$ and $g_1$ have no
irreducible factors involving $z$ alone, we may conclude they are
constant multiples of one another.  The constant can be absorbed into
the definition of $h_1$ so that $p_1(z,w)=h_1(z) g(z,w)$.  Then,
$|p_1|^2 = |p|^2$ on $\T^2$ implies that $|h_1|^2 = |q|^2$ on $\T$.
It is then elementary to show $h_1$ is obtained by flipping some of
the roots of $q$ to inside $\D$.
\end{proof}

\begin{lemma}
 If the split-shift orthogonality condition holds with a given
 split-poly $p$, then the spaces $\mcK_1$ and $\mcK_2$ are uniquely
 determined by $p$.
 \end{lemma}

 \begin{proof}
 Looking at the formulas in Theorem \ref{aglerdecomps}, we see that
 since all of the $E$ or $F$ kernels are uniquely determined, the
 kernels $\rev{K}_2 - K_1$ and $w\bar{\eta}\rev{K}_2 - K_1$ are
 uniquely determined.  We see that $(1-w\bar{\eta})K_1$ is uniquely
 determined and so $K_1$ is uniquely determined.  A similar argument
 shows $K_2$ is uniquely determined.
 \end{proof}

We can now give a description of all possible shift-splits.

 \begin{prop} \label{allshiftsplits}
If the split-shift condition holds for a positive linear form $\mcT$
on $\mcL_{n,m}$, then there exists $g \in \C[z,w]$ with no zeros on
$\T\times \cD$ and no irreducible factors involving $z$ alone and
stable $q\in \C[z]$, such that $q(z)g(z,w)$ is a split-poly.  Write
$n_1:=\deg_z g$ and $n_0:= n-n_1$.  Every other split-poly is of the
form 
\[
q_1(z)q_2(z)g(z,w)
\]
where $\deg q_1q_2 \leq n_0$, $q(z) = q_1(z)\rev{q}_2(z)$, and
$\rev{q}_2$ is reflected at the degree of $q_2$.  The associated
shift-split $(\mcK_1, \mcK_2)$ is given by
\[
\mcK_1 = \vee P_{\mcE_{n-1,m}^1} \{ z^j q_1(z)g_1(z): 0\leq j < \deg
q_2 + \deg g_2\}
\]
\[
\mcK_2 = \vee P_{\mcE_{n-1,m}^1} \{ z^j q_2(z) g_2(z): 0\leq j < n- \deg
q_2 - \deg g_2\}
\]
where $g(z,0) = g_1(z)g_2(z)$ with $g_1$ having no zeros in $\cD$ and
$g_2$ having all zeros in $\D$.
\end{prop}

Proposition \ref{mctoss} singles out the shift-split with minimal
$\mcK_2$ which would correspond to the split-poly $\rev{q}(z)g(z,w)$,
where $\rev{q}(z)$ is reflected at degree $n_0$.  The shift-split with
minimal $\mcK_1$ corresponds to split-poly $q(z)g(z,w)$.  This leads to
a canonical decomposition of $\mcE_{n-1,m}^1$ which does not depend on
a choice of shift-split. Let $\deg g = (n_1,m)$ and $n_0 := n-n_1$.

Define
\[
\mcK_{0} = \vee\{z^j g(z,w): 0\leq j< n_0\}
\]
\[
\mathcal{A} = \vee P_{\mcE_{n-1,m}^1} \{ z^j q(z)g_1(z): 0\leq j < \deg g_2\}
\]
\[
\mathcal{B} = \vee P_{\mcE_{n-1,m}^1} \{ z^j \rev{q}(z) g_2(z): 0\leq
j < n_1-\deg g_2 \}.
\]
\begin{theorem} \label{ABdecomp} 
Let $\mcT$ be a positive linear form on $\mcL_{n,m}$ satisfying the
split-shift condition. Then,
\[
\mcE_{n-1,m}^1 = \mcK_{0} \oplus\mathcal{A} \oplus \mathcal{B}.
\]
Both $(\mcK_0 \oplus \mathcal{A}, \mathcal{B})$ and $(\mathcal{A},
\mcK_0\oplus \mathcal{B})$ are shift-splits. If
  $(\mcK_1,\mcK_2)$ is any shift-split, then $\mathcal{A} \subset
  \mcK_1$ and $\mathcal{B} \subset \mcK_2$.
\end{theorem}

\begin{proof}
It follows by inspection of definitions that $\mathcal{A}
  \subset \mcK_1$ and $\mathcal{B} \subset \mcK_2$ using $\mcK_1$ and
  $\mcK_2$ from Proposition \ref{allshiftsplits}.

Let $g\in \C[z,w]$ and $q\in \C[z]$ be as in the previous proposition.
In $L^2(1/|qg|^2d\sigma)$, $g \perp z^jw^{k+1}$ for $j\in \Z$ and
$k\geq 0$ because
\[
\ip{z^jw^{k+1}}{g} = \int_{\T} \frac{z^j}{|q(z)|^2} \int_{\T}
\frac{w^{k+1}}{g(z,w)} \frac{|dw| |dz|}{(2\pi)^2} = 0
\]
since $1/g(z,\cdot)$ is holomorphic.  Therefore, $z^jg(z,w) \in
\mcE_{n-1,m}^1$ for $0\leq j < n_0$ and we see
\[
\mcK_{0} = \vee P_{\mcE_{n-1,m}^1} \{ z^j g_1(z)g_2(z): 0\leq j <
n_0\} 
\]
is contained in 
\[
\vee P_{\mcE_{n-1,m}^1} \{z^j g_1(z): 0 \leq j < n_0+\deg g_2\}
\]
but this corresponds to $\mcK_1$ in the shift-split coming from the
split-poly $\rev{q}(z)g(z,w)$.  Hence, this space and $\mcK_0$ must be
orthogonal to the associated $\mcK_2$ which happens to be
$\mathcal{B}$.  Notice also that 
\[
\mathcal{A} \subset \vee P_{\mcE_{n-1,m}^1} \{z^j g_1(z): 0 \leq j <
n_0+\deg g_2\}.
\]

A similar argument shows $\mcK_0$ is orthogonal to $\mathcal{A}$, and
by dimension considerations
\[
\mcK_0 \oplus \mathcal{A} = \vee P_{\mcE_{n-1,m}^1} \{z^j g_1(z): 0 \leq j <
n_0+\deg g_2\}
\]
and again by dimension considerations
\[
\mcK_0 \oplus \mathcal{A} \oplus \mathcal{B} = \mcE_{n-1,m}^1.
\]

We already noted that $\mcK_0\oplus \mathcal{A}$ corresponds to
``$\mcK_1$'' in some shift-split.  Therefore, $(\mcK_0\oplus
\mathcal{A}, \mathcal{B})$ is a shift-split.  By a similar argument,
$(\mathcal{A}, \mcK_0 \oplus \mathcal{B})$ is a shift-split.

\end{proof}

Propositions \ref{sstomc} and \ref{mctoss} together show that the
invariant subspace of $T$ generated by the range of $B$ is the minimal
possible ``$\mcK_2$'' occurring in a shift-split.  We have already
computed the minimal $\mcK_2$, which is $\mathcal{B}$.  A similar
argument can be used for the minimal $\mcK_1$ which is
$\mathcal{A}$. This implies the following.

\begin{theorem} \label{BAformulas}
If the split-shift condition holds, 
\[
\mathcal{B} = \vee\{T^j Bf: f\in w \mcF_{n,m-1}^2, j=0,1,\dots,n-1\}
\]
and
\[
\mathcal{A} = \vee\{(T^*)^j A^* f: f \in w\mcE_{n,m-1}^2,
j=0,1,\dots, n-1\}
\]
where $A^* = P_{\mcE_{n-1,m}^1} M_{1/z}:w\mcE_{n,m-1}^2\to
\mcE_{n-1,m}^1$  and $T^* = P_{\mcE_{n-1,m}^1} M_{1/z}:\mcE_{n,m-1}^1\to
\mcE_{n-1,m}^1$. 
\end{theorem}


We can now prove the stratified characterization of
Bernstein-Szeg\H{o} measures.

\begin{proof}[Proof of Corollary \ref{corstrat}]
Suppose $\mcT$ is a positive linear form on $\mcL_{n,m}$ given by 
\[
\mcT(z^jw^k) = \int_{\T^2} z^jw^k
\frac{|dz||dw|}{(2\pi)^2|p(z,w)|^2}, \qquad |j|\leq n, |k| \leq m,
\]
where $p \in \C[z,w]$ has no zeros in $\T\times \cD$, degree at most
$(n,m)$ and $p(z,0)$ has $d$ zeros in $\D$.  We write $p(z,0) = a(z)
b(z)$ where $b$ has all zeros in $\D$ and $a$ has no zeros in $\cD$.
By Theorem \ref{BStoSS}, $\mcT$ possesses a shift-split
$(\mcK_1,\mcK_2)$ where $\mcK_1$ has dimension $d=\deg b(z)$.

Next, supposing $\mcT$ possesses a split-shift $(\mcK_1,\mcK_2)$ where
$\mcK_1$ has dimension $d$, by Theorems \ref{ABdecomp} and
\ref{BAformulas} we have $\mathcal{A} \subset \mcK_1 \subset
\mathcal{A}\oplus \mcK_0$. Therefore,
\begin{equation} \label{dimbounds}
\dim \mathcal{A} \leq d \leq n - \dim \mathcal{B}.
\end{equation}

Finally, if $\mcT$ satisfies the matrix condition and
\eqref{dimbounds}, then we see from Proposition \ref{allshiftsplits}
that it is possible to choose $\mcK_1$ with dimension $d$ and the
corresponding split-poly $p$ has the desired property that $p(z,0)$
has $d$ roots in $\D$.
\end{proof}

\section{Construction of $p$ from Fourier coefficients}  \label{constructp}
In this section, it is useful to write $z=(z_1,z_2)$ for an element of
$\C^2$ as opposed to $(z,w)$, so that we can use multi-index notation
$z^u = z_1^{u_1} z_2^{u_2}$.

Supposing the split-shift condition does hold, how do we construct $p$
directly from the Fourier coefficients
\[
\mcT(z^{-u}) = c_u \text{ ?}
\]
In principle, one could construct $(\mcK_1,\mcK_2)$ and then produce
$p$ as an element of $\mcE_{n,m} \ominus (\mcK_1\oplus z\mcK_2)$;
however, this is quite involved.  In this section we describe a
simpler procedure assuming we already know that the shift-split
condition holds. 

First, we construct an orthonormal basis for $\mcE_{n,m}^2$. It helps
to use interval notation for subsets of integers as in $[0,n] =
\{0,\dots, n\}$. Let $S_j = [0,n]\times[0,m] \setminus\{(0,0),\dots,
(0,j-1)\}$, $S_0=[0,n]\times [0,m]$.  Let
\[
(\gamma^{(j)}_{u,v})_{u,v \in S_j} = (c_{v-u})^{-1}_{u,v\in S_j}
\]
and define
\[
\phi_j(z) = \sum_{v \in S_j} \gamma^{(j)}_{(0,j),v} z^v/ \sqrt{\gamma^{(j)}_{(0,j),(0,j)}}.
\]
Then, $\phi_0,\phi_1,\dots, \phi_m$ form an orthonormal basis for
$\mcE_{n,m}^2$.  To see this let $u \in S_{j+1}$ and $\gamma =
\sqrt{\gamma^{(j)}_{(0,j),(0,j)}}$.  We compute
\[
\ip{\phi_j}{z^u} = \sum_{v \in S_j} \gamma^{(j)}_{(0,j),v}
c_{u-v}/\gamma = \delta_{(0,j),u}/\gamma = 0
\]
since $(0,j) \notin S_{j+1}$ and $S_{j+1} \subset S_j$.  For $k>j$,
$\phi_k$ is a combination of $z^u$ with $u \in S_k \subset S_{j+1}$
and therefore $\phi_k \perp \phi_j$ for $k>j$.  Also,
\[
\ip{\phi_j}{\phi_j} = \sum_{v,u \in S_j} \gamma^{(j)}_{(0,j),v}c_{u-v}
\bar{\gamma}^{(j)}_{(0,j),u}/\gamma^2  = \sum_{u\in S_j}
\delta_{(0,j),u}\bar{\gamma}^{(j)}_{(0,j),u}/\gamma^2 =
\gamma^{(j)}_{(0,j),(0,j)}/\gamma^2 = 1.
\]

 The reproducing kernel for $\mcE_{n,m}^2$ is therefore
 $E_m^2(z;\zeta) = \sum_{j=0}^{m} \phi_j(z)\overline{\phi_j(\zeta)}$.

We assume $p(z_1,z_2)=q(z_1)g(z_1,z_2)$ with $q$ stable and $g$ has no
factors with $z_1$ alone and then show how to construct $g$ and $q$
using only the moments $c_u$.  Theorem \ref{aglerdecomps} proves
\[
 z_1^n E^2_m(z;1/\bar{z}_1,0)= p(z)z_1^n\bar{p}(1/z_1,0) = q(z_1)g(z)
 z_1^n \bar{q}(1/z_1) \bar{g}(1/z_1,0).
\]
From this we can calculate $g$ up to a constant multiple.  The key
point is that the product of all factors of the above polynomial that
involve $z_1$ alone will be the greatest common divisor of the
coefficients of powers of $z_2$.

Let us write
\[
z_1^n E^2_m(z;1/\bar{z}_1,0) =  \sum_{j=0}^{m} E_j(z_1) z_2^j,
\]
and then compute $Q = \gcd \{E_0,E_1,\dots, E_m\}$ using the Euclidean
algorithm.  Then, $Q(z_1) = C q(z_1)z_1^n \bar{q}(1/z_1)
\bar{g}(1/z_1,0)$ for some constant $C$.  This gives
\[
z_1^n E^2_m(z;1/\bar{z}_1,0)/Q(z_1) = g(z)
\]
possibly with a constant.  At this stage we look at the one variable
moment problem
\[
c_j = \mcT(z_1^{-j}g(z) \bar{g}(1/z_1,1/z_2)) = \int_{\T} z_1^{-j}
\frac{|dz_1|}{2\pi|q(z_1)|^2}
\]
for $|j|\leq n_0:=n-\deg_z g$.  Set
\[
\gamma_{j,k} = (c_{k-j})^{-1}_{j,k \in [0,n_0]}
\]
and then we can construct 
\[
q(z_1) = \sum_{j=0}^{n_0} \gamma_{0,j} z_1^j/\sqrt{\gamma_{0,0}}
\]
(up to a unimodular multiple) by one variable theory.    

Hence, we have constructed $p$ as $p(z) = q(z_1)g(z)$.

\section{Applications} \label{applications}

\subsection{Autoregressive filters} \label{autoreg}
A direct application of the above work is to two variable
autoregressive models \cite{rosen}. 

We consider (wide sense) stationary processes $X=(X_{u})_{u\in\Zset^2}
$ depending on two discrete variables defined on a fixed probability
space $( \Omega , {\mathcal A} , P )$.  We shall assume that $X$ is a
{\em zero mean} process, i.e.\! the means $E (X_{u} )$ are equal to
zero. Recall that the space $L^2 ( \Omega , {\mathcal A} , P )$ of
square integrable random variables endowed with the inner product
 $$ \langle X , Y \rangle := E ( X Y^* ) $$
 is a Hilbert space. A sequence $X = ( X_{m} )_{m \in \Zset^2}$ is called
 a {\em stationary process}
 on $\Zset^2$ if for $m,n \in \Zset^2$
 we have that $$
 E ( X_{m} X^*_{n } ) = E( X_{m+u} X_{n+u}^* ) =:
 R_X(m-n )
 , \ \hbox{\rm for all} \ u \in \Zset^2 . $$
 It is known that the function $R_X$, termed the {\em covariance
 function} of $X$,
 defines a {\em positive semi-definite function} on $\Zset ^2$, i.e.
 $$ \sum_{i,j=1}^k \alpha_i \bar{\alpha}_j R_X(r_i - r_j) \ge 0,$$ 
for
 all $k \in {\Nset}$, $\alpha_1 ,\dots,\alpha_k \in \Cset,
 r_1,\dots,r_k \in \Zset^2$ and Bochner's Theorem states that for such
 a function $R_X$ there is a positive regular bounded measure $\mu_X$
 defined for Borel sets on the torus $[0,2 \pi ]^2$ such that
 $$ R_X(u) = \int e^{-i \langle u,t\rangle } d \mu_X (t), $$ for all
 two tuples of integers $u$.  The measure $\mu_X$ is referred to as
 the {\em spectral distribution measure}\index{Spectral distribution}
 of the process $X$. The {\em spectral density}\index{Spectral
   density} $f_X (t)$ of the process $X$ is the spectral density of
 the absolutely continuous part of $\mu_X$, i.e.\! the absolutely
 continuous part of $\mu_X$ equals
 $$ f_X(t_1, t_2) \frac{dt_1 dt_2}{(2 \pi)^2} . $$

  Let $\tilde H=\{(k,l):-\infty<k<\infty, l>0\} \cup \{(k,0),k>0\}$ and let
  $\Lambda_{n,m}=\{(k,l):0\leq k\leq n, 0\leq l\leq m \}\subset \tilde H\cup \{ (0,0)
  \}$ be a finite set.  A zero-mean stationary stochastic process $X=
  (X_{u})_{u\in \Zset^2}$ is said to be {\it extended autoregressive}
  or $\eAR(n,m)$, if there exist complex numbers $a_{k} , k \in
  \Lambda_{n,m}$ with $a_{(i,0)}\ne0\ {\rm for\ some}\ 0\le i\le n$,
  so that for every $u$
 \begin{equation}
 \label{areq}
 \mathop{\sum\limits_{k\in\Lambda_{n,m} }}
  a_{k}X_{u-k}=\mathcal{E}_{u}
 ,\qquad u\in\Zset^2,
 \end{equation}
 where $\{\mathcal{E}_u\ : u \in\Zset^2 \}$ is a white noise zero mean
 process with variance 1.  The $\eAR(n,m)$ process is said to be {\it
   acausal (in z)} if there is a solution to equations \eqref{areq} of
 the form
 $$ X_{u} = \mathop{\sum\limits_{k\in \tilde H\cup\{
     (0,0),(-1,0),\ldots\}}} \phi_{k}\mathcal{E}_{u-k}, u \in \Zset^2
 ,
 $$ with $\mathop{\sum\limits_{k\in \tilde H \cup\{ (0,0),(-1,0)\ldots
     \}}}| \phi_{k}|<\infty$ and it is said to be {\it causal } if
 there is a solution of the form
$$ X_{u} = \mathop{\sum\limits_{k\in \tilde H\cup\{ (0,0)\} }}
 \phi_{k}\mathcal{E}_{u-k}, u \in \Zset^2 ,
 $$ with $\mathop{\sum\limits_{k\in \tilde H \cup\{ (0,0)\}}}|
 \phi_{k}|<\infty.$ From the general theory of autoregressive models
 it follows that if \eqref{areq} has a causal (acausal (in $z$))
 solution then
\begin{equation}\label{auto}
p(z,w)= \sum_{v\in\Lambda_{n,m}} a_{v} z^{v_1} w^{v_2}.
\end{equation}
is stable on $\bar \D^2$ ($\T \times\cD$).


The bivariate extended autoregressive ($\eAR$) model problem concerns
 the following.
 Given autocorrelation elements
 $$c_{k}=E(X_{0} X_{k}^{*}),\ k\in\Lambda_{n,m}-\Lambda_{n,m}$$
 determine, if possible, the coefficients $a_l, l \in \Lambda_{n,m}$
 of an acausal autoregressive filter representation. In \cite{GW04}
 necessary and sufficient conditions were given for the
 autocorrelation coefficients in order for the $\eAR(n,m)$ to have a
 causal solution. Here we give necessary and sufficient conditions in
 order for an $\eAR(n,m)$ model to have an acausal solution.
 
If we begin with a polynomial that is nonzero for $(z,w)\in
\T\times\cD$ then choosing the autoregressive filter coefficients as
in equation \eqref{auto} give an $\eAR(n,m)$ model whose Fourier
coefficients give a linear form that is positive and satisfies
conditions in Theorem~\ref{mainthm}. Conversely, we can use the
conditions in Theorem~\ref{mainthm} to characterize the existence of
acausal (in $z$) solution.

\begin{theorem}
 \label{autor}
 Given autocorrelation elements
 $c_{k,l},\ (k,l)\in\Lambda_{n,m}-\Lambda_{n,m}$ there exists an
 acausal (in $z$) solution to the $\eAR(n,m)$ problem if and only if the
 linear form $\mcT$ determined by the Fourier coefficients $c_{k}$ is
 positive and satisfies one of the equivalent conditions of Theorem
 \ref{mainthm}
 \end{theorem}

\begin{corollary} 
With the hypotheses of the above Theorem there exists a casual
solution to the $\eAR(n,m)$ problem if and if the linear form $\mcT$
determined by the Fourier coefficients $c_k$ is positive and $A=0$.
\end{corollary}

\subsection{Full measure characterization} \label{fullmeasure}
We now identify which measures $d\mu$ are of the form
\[
\frac{1}{|p(z,w)|^2}d\sigma
\]
where $p \in \C[z,w]$ has no zeros on $\T\times \cD$.

Corollary \ref{cor:neccond} provides necessary conditions which can be
encoded as
\begin{equation} \label{neccond}
\mcE_{n,M}^2 = \mcE_{n+j,M}^2
\end{equation}
for all $j\geq 0$ and $M\geq m-1$.  By performing the reflection
operation, it follows that
\[
z^j\mcF_{n,M}^2 = \mcF_{n+j,M}^2
\]
for all $j \geq 0, M\geq m-1$.

It turns out that conditions \eqref{neccond} for $M=m-1,m$ are
sufficient to show that the moments $\int z^jw^k d\mu$ agree with the
moments of a Bernstein-Szeg\H{o} measure when $j\in \Z$ and $|k|\leq
m$ (i.e. on a strip).  It is then another issue to prove that the
Bernstein-Szeg\H{o} measure obtained with a particular $m$ agrees with
other choices.

Fix $m$ and define $A_N = P_{w\mcE_{N,m-1}^2} M_zP_{\mcE_{N-1,m}^1}$,
$T_N=P_{\mcE_{N-1,m}^1} M_z P_{\mcE_{N-1,m}^1}$, and $B_N =
P_{\mcE_{N-1,m}^1}P_{w\mcF_{N,m-1}^2}$.

\begin{lemma}
Assume \eqref{neccond} holds for $j\geq 0$ and for  $M=m-1,m$. Then,
$A_N T_N^k B_N = 0$ for $N\geq n$, $k\geq 0$.   
\end{lemma}

\begin{proof}
Let $P=P_1+P_2$ be the projection onto the space
\[
\mcP_{N+k,m} \ominus w\mcP_{N-1,m-1} = \mcE_{N-1,m}^{1}
\bigoplus_{j=0}^{k} \mcF_{N+j,m}^2 
\]
where $P_1,P_2$ are the projections onto $\mcE_{N-1,m}^{1},
\bigoplus_{j=0}^{k} \mcF_{N+j,m}^2$ respectively.  

Noting that $z\bigoplus_{j=0}^{k} \mcF_{N+j,m}^2 =
\bigoplus_{j=0}^{k}\mcF_{N+j+1,m}^2$ by \eqref{neccond}, we have
$P_1M_zP_2 = 0$.  Then, $T_N = P_1 M_z P_1 = P_1 M_z (P_1+P_2) =
P_1M_z P$.  Therefore, for $k\geq 1$ we have
\[
T_N^k = P_1 M_z (P M_z P)^{k-1}P.
\]
Similarly, $A_N = P_{w\mcE_{N,m-1}^2} M_z P_1 = P_{w\mcE_{N,m-1}^2} M_z
P$, so that
\[
A_N T_N^k = P_{w\mcE_{N,m-1}^2} M_z (PM_zP)^kP.
\]
Next, $P P_{w\mcF_{N,m-1}^2} = P_{w\mcF_{N,m-1}^2}$ while
\[
PM_zP_{w\mcF_{N+j,m-1}^2} = P_{w\mcF_{N+j+1,m}^2} M_z P_{w\mcF_{N+j,m-1}^2}
\]
 since $zw\mcF_{N+j,m-1}^2 = w\mcF_{N+j+1,m-1}^2$ by \eqref{neccond} so
 that inductively we have
\[
A_N T_N^k B_N = P_{w\mcE_{N,m-1}^2} M_z P_{w\mcF_{N+k,m-1}^2} \prod_{j=0}^{k-1} \left(M_z P_{w\mcF_{N+j,m-1}^2}\right)
\]
where the product is multiplied from right to left as $j$ goes from
$0$ to $k-1$ (if $k=0$, the product is $I$).  But, $P_{w\mcE_{N,m-1}^2}
M_z P_{w\mcF_{N+k,m-1}^2} = 0$ as $zw\mcF_{N+k,m-1}^2 =
w\mcF_{N+k+1,m-1}^2 \perp w\mcE_{N,m-1}^2$.  
\end{proof}

Therefore, assuming \eqref{neccond} for $M=m,m-1$ and $j\geq 0$, the
matrix condition holds for the positive linear form on
$\mcL_{N,m}$ for $N\geq n$. So for each $N$, there is a $p_N\in
\C[z,w]$ of degree at most $(N,m)$ with no zeros on $\T\times \cD$
such that the Bernstein-Szeg\H{o} measure for $p_N$ matches the
moments of $d\mu$ on $\mcL_{N,m}$.  We can further assume that each
$p_N$ has been normalized so that $p_N(z,w) = q_N(z) g_N(z,w)$ where
$q_N$ is stable in $z$ and $g_N$ has no factors involving $z$ alone.
By Theorem \ref{aglerdecomps} and \eqref{neccond}, if we set $z=\zeta
\in \T$, $w \in \C, \eta = 0$, we get
\[
p_n(z,w)\overline{p_n(z,0)} = p_{n+j}(z,w)\overline{p_{n+j}(z,0)}
\]
for $j\geq 0$.  This implies $g_n = g_{n+j}$ for each $j\geq 0$ (after
absorbing constants into $q$'s if necessary) and then by stability of
each $q_N$, $q_n = q_{n+j}$ for each $j\geq 0$.  Therefore, the
moments of $d\mu$ on the strip $\{z^jw^k: j \in \Z, |k|\leq m\}$ are
matched by those of $1/|p_n|^2 d\sigma$.

\begin{theorem}
Let $d\mu$ be a positive Borel measure.  
If
\[
\mcE_{n,m}^2 = \mcE_{n+j,m}^2 \quad \mcE_{n,m-1}^2 = \mcE_{n+j,m-1}^2
\]
for $j\geq 0$, then there exists $p\in \C[z,w]$ of degree at most
$(n,m)$ with no zeros on $\T\times \cD$ such that
\[
\int z^j w^k d\mu = \int \frac{z^jw^k}{|p(z,w)|^2} d\sigma
\]
for $j\in \Z$ and $|k|\leq m$.
\end{theorem}

To get the full measure characterization, we note that for a
Bernstein-Szeg\H{o} measure (and recalling $G_{\eta}$ from Section
\ref{basic})
\[
G_0(z,w) =  p(z,w) z^n\bar{p}(1/z,0)
\]
is an element of the one dimensional space $\mcH_m$ defined in \eqref{HM},
but by all of the orthogonality relations for Bernstein-Szeg\H{o}
measures it is also in $\mcH_{M}$ for $M\geq m$.
Therefore, a set of necessary conditions is 
\[
\mcH_m = \mcH_{m+j} \text{ for } j \geq 0.
\]

\begin{theorem}
Let $d\mu$ be a positive Borel measure on $\T^2$ satisfying
\[
\mcE_{n,M}^2 = \mcE_{n+j,M}^2 
\text{ and } \mcH_{m} = \mcH_{m+j}
\]
for $M\geq m-1$ and $j\geq 0$.  Then, there exists $p\in \C[z,w]$ of
degree at most $(n,m)$ with no zeros on $\T\times \cD$ such that
\[
d\mu = \frac{d\sigma}{|p(z,w)|^2}
\]
\end{theorem}

\begin{proof}
The conditions on $\mcE_{\cdot,\cdot}^2$ imply that for each $M\geq
m$, there exists $p_M \in \C[z,w]$ of degree at most $(n,M)$ with no
zeros on $\T\times \cD$ such that the moments of $d\sigma/|p_M|^2$ match
those of $d\mu$ on the strip $\{z^jw^k: j\in \Z, |k|\leq M\}$.  We
normalize $p_M(z,w) = q_M(z)g_M(z,w)$ where $q_M$ is stable and $g_M$
has no factors involving $z$ alone.

Using the assumption $\mcH_m = \mcH_M$ for $m\geq M$, it follows that
for each $m\geq M$
\[
p_m(z,w)z^n\bar{p}_m(1/z,0) = C p_M(z,w)z^n\bar{p}_M(1/z,0)
\]
for some constant $C$.  We then must have that $g_m$ and $g_M$ are
constant multiples and then since $q_m, q_M$ are stable, they too must
be constant multiples of one another.  Therefore, $p_m$ and $p_M$ must
be constant multiples.  The constant must be unimodular since $p_m$
and $p_M$ have unit norm.  Therefore, the measures $1/|p_m|^2 d\sigma
= 1/|p_M|^2d\sigma$ match all of the moments of $d\mu$.  Hence,
$1/|p_m|^2 d\sigma = d\mu$.
\end{proof}

\subsection{Concrete expression for the full measure characterization} \label{concretefullmeasure}
The conditions
\[
\mcE_{n,M}^2 = \mcE_{n+j,M}^2 
\text{ and } \mcH_{m} = \mcH_{m+j}
\]
for $M\geq m-1$ and $j\geq 0$ given above can be written directly in
terms of the Fourier coefficients of $\mu$ 
\[
c_u = \int z^{-u} d\mu \qquad u=(u_1,u_2) \in \Z^2
\]
as follows.  Similar to Section \ref{constructp} it is useful to write
$z=(z_1,z_2)$ for an element of $\C^2$ as opposed to $(z,w)$ and we
will use multi-index notation $z^u = z_1^{u_1}z_2^{u_2}.$ Also, $[0,N]
= \{0,1,\dots,N\}$; there should be no confusing this with a closed
interval of real numbers.

Let 
\[
(\gamma^{N,M}_{u,v})_{u,v \in [0,N]\times[0,M]} = (c_{v-u})^{-1}_{u,v
  \in [0,N]\times[0,M]}.
\]
A basis for $\mcE_{N,M}^2$ consists of
\[
f^{N,M}_{j}(z) = \sum_{v \in [0,N]\times[0,M]} \gamma^{N,M}_{(0,j),v} z^v
\]
for $j=0,1,\dots, M$.  In order for $\mcE_{N+1,M}^2 = \mcE_{N,M}^2$ to
hold we need the coefficients of $z_1^{N+1}z_2^k$ for $k=0,1,\dots, M$
to vanish in $\mcE_{N+1,M}^2$.  Looking at $f^{N+1,M}_{j}$ this amounts
to
\[
\gamma^{N+1,M}_{(0,j),(N+1,k)} = 0
\]
for $j,k=0,1,\dots, M$.

Therefore, the conditions $\mcE_{n,M}^2 = \mcE_{n+j,M}^2$ for $j\geq 0$
and $M\geq m-1$ can be expressed as
\[
\gamma^{N+1,M}_{(0,j),(N+1,k)} = 0
\]
for $N\geq n$, $M\geq m-1$, $j,k = 0,1,\dots, M$.

Next we turn to the conditions 
\[
\mcH_{m} = \mcH_{m+j} \text{ for } j \geq 0.
\]
Recall
\[
\mcH_{M} = \mcP_{2n,M} \ominus \vee\{z_1^jz_2^k: 0\leq j\leq 2n, 0\leq k
\leq M, (j,k) \ne (n,0)\}
\]
so that a nonzero element of the one dimensional space $\mcH_M$ is
given by
\[
g_M(z) = \sum_{u \in [0,2n]\times [0,M]} \xi^{M}_{(n,0),u} z^u
\]
where we define
\[
\xi^{M}_{u,v} = (c_{v-u})_{u,v\in [0,2n]\times[0,M]}^{-1}.
\]

The condition $\mcH_{M} = \mcH_{M-1}$ can then be expressed as
\[
\xi^{M}_{(n,0),(j,M)} = 0 \qquad j=0,1,\dots, 2n.
\]

Let us summarize everything.

\begin{theorem}
Let $\mu$ be a positive, finite measure on $\T^2$ with moments $c_u$
for $u\in \Z^2$.  There exists a polynomial $p\in \C[z,w]$ of degree at
most $(n,m)$ with no zeros on $\T\times \cD$ such that
\[
d\mu = \frac{1}{|p(z,w)|^2} d\sigma
\]
if and only if
\begin{enumerate}
\item for all $N,M\geq 0$,
\[
\det(c_{v-u})_{u,v \in [0,N]\times[0,M]} \ne 0
\]

\item for $N\geq n$, $M\geq m-1$, $j,k = 0,1,\dots, M$
\[
\gamma^{N+1,M}_{(0,j),(N+1,k)} = 0
\]
and 
\item for $M\geq m$, $j=0,1,\dots, 2n$
\[
\xi^{M}_{(n,0),(j,M)} = 0 
\]
\end{enumerate}
where
\[
(\gamma^{N,M}_{u,v})_{u,v \in [0,N]\times[0,M]} = (c_{v-u})^{-1}_{u,v
  \in [0,N]\times[0,M]}
\]
\[
\xi^{M}_{u,v} = (c_{v-u})_{u,v\in [0,2n]\times[0,M]}^{-1}.
\]
\end{theorem}

\section{Generalized distinguished varieties}\label{GDV}

\subsection{Construction of the sums of squares
  formula} \label{sec:kummert} Here we use Kummert's approach as in
\cite{aK89} to give a different proof of the sums of squares formula
Theorem \ref{sosthm}.  One advantage of this approach is that it works
for $p$ with no zeros on $\T\times \D$ and no factors in common with
$\rev{p}$ (rather than assuming no zeros on $\T\times \cD$).  This
approach is also useful because it shows how to compute the
reproducing kernels in the decomposition of $p$ using only one
variable theory.

\begin{theorem} \label{gensosthm}
Suppose $p \in \C[z,w]$ has degree $(n,m)$, no zeros on $\T\times \D$
and no factors in common with $\rev{p}$.  Let $n_2$ be the number of
zeros of $p(z,0)$ in $\D$ and $n_1 = n-n_2$. Then, there exist vector
polynomials $E \in \C^m[z,w], A \in \C^{n_1}[z,w], B \in
\C^{n_2}[z,w]$ such that
\begin{itemize}
\item 
\[
|p(z,w)|^2 - |\rev{p}(z,w)|^2 = (1-|w|^2)|E(z,w)|^2 +
(1-|z|^2)(|A(z,w)|^2 - |B(z,w)|^2).
\]
\item $E$ has degree at most $(n,m-1)$ and $A$ and
$B$ have degree at most $(n-1,m)$, and
\item the entries of $A$ and $B$ form a linearly independent set of
  polynomials.
\end{itemize}
\end{theorem}

The last two details are needed in Section \ref{detreps}.

 Consider for $z \in \T$
\[
\frac{p(z,w)\overline{p(z,\eta)} -
  \rev{p}(z,w)\overline{\rev{p}(z,\eta)}}{1-w\bar{\eta}} =
(1,\bar{\eta},\dots, \bar{\eta}^{m-1}) T(z) (1,w,\dots,  w^{m-1})^{t}
\]
where $T(z)$ is an $m\times m$ matrix valued trigonometric polynomial
which is positive definite for all but finitely many values of $z \in
\T$.  This is because $T(z)$ is positive definite for each value of
$z$ such that $p(z,\cdot)$ has no zeros in $\T$.  There can only be
finitely many zeros on $\T^2$ or else $p$ and $\rev{p}$ would have a
common factor.  Let $S = \{z\in \T: \det T(z) = 0\}$.

By the matrix Fej\'er-Riesz theorem in one variable, we may factor
\[
T(z) = E(z)^* E(z)
\]
where $E$ is an invertible matrix polynomial on $\D$ of degree at most
$n$.  Let
\[
E(z,w) = E(z) (1,w,\dots,w^{m-1})^{t}
\]
so that 
\[
p(z,w)\overline{p(z,\eta)} -
  \rev{p}(z,w)\overline{\rev{p}(z,\eta)} =
  (1-w\bar{\eta})E(z,\eta)^* E(z,w)
\]
for $z \in \T$.  We are using both the notations $E(z)$ and $E(z,w)$,
but no confusion should arise.

Then, for fixed $z \in \T\setminus S$, the map which maps
\[
\begin{pmatrix} p(z,w) \\ wE(z,w) \end{pmatrix}
\mapsto \begin{pmatrix} \rev{p}(z,w) \\ E(z,w) \end{pmatrix}
\]
extends to a unitary $U(z)$ which we can explicitly solve for.
Write \[
p(z,w) = \sum_{j=0}^{m} p_j(z)w^j, \qquad \rev{p}(z,w) =
\sum_{j=0}^{m} \rev{p}_{m-j}(z) w^j. 
\]
Then,
\[
U(z) =  \begin{pmatrix} 
\begin{matrix} \rev{p}_m(z) & \cdots & \rev{p}_1(z) \end{matrix} 
& \rev{p}_0(z) \\ 
E(z) & 0 \end{pmatrix} \begin{pmatrix} p_0(z) & \begin{matrix} p_1(z) & \cdots &
    p_m(z) \end{matrix} \\
0 & E(z) \end{pmatrix}^{-1} 
\]
which is unitary by construction for $z \in \T\setminus S$, but
clearly extends to a matrix rational function with poles in $\D$ at
the zeros in $\D$ of $p_0(z)$.  Moreover, any singularities on $\T$
must be removable because $U$ is bounded on a punctured neighborhood
in $\T$ of each singularity.

Set $n_2$ to be the number of zeros of $p(z,0)$ in $\D$ and
$n_1=n-n_2$.  Theorem \ref{Udecomp} and Section \ref{Uproperties}
below prove that
\begin{equation} \label{posnegsos}
\frac{I-U(\zeta)^* U(z)}{1-\bar{\zeta}z} = F(\zeta)^*F(z) -
G(\zeta)^*G(z)
\end{equation}
where $F$ is $n_1 \times (m+1)$ and $G$ is $n_2\times (m+1)$, and the
rows of $F$ and $G$ are linearly independent as vector functions;
meaning there is no non-zero solution $(v_1,v_2) \in \C^{n}$ to
\begin{equation} \label{detail}
v_1F(z) + v_2 G(z) \equiv 0.
\end{equation}
Accepting all of this for now, we rearrange \eqref{posnegsos} to get
\[
I +  \bar{\zeta}F(\zeta)^*zF(z) + G(\zeta)^* G(z) = U(\zeta)^* U(z) +
F(\zeta)^*F(z) + \bar{\zeta}G(\zeta)^* zG(z)
\]
and so there exists an $(m+1+n)\times (m+1+n)$ unitary (with two
indicated block decompositions)
\begin{equation} \label{Vblocks}
V
= \begin{matrix} & \begin{matrix} \C^{m+1} & \C^{n} \end{matrix} \\
\begin{matrix} \C^{m+1} \\ \C^n \end{matrix} & \begin{pmatrix} V'_1 &
  V'_2 \\ V'_3 & V'_4 \end{pmatrix} \end{matrix}= 
\begin{matrix} & \begin{matrix} \C^{1} & \C^{m+n} \end{matrix} \\
\begin{matrix} \C^{1} \\ \C^{m+n} \end{matrix} & \begin{pmatrix} V_1 &
  V_2 \\ V_3 & V_4 \end{pmatrix}
\end{matrix}
\end{equation}
 such that
\begin{equation} \label{VFG}
V\begin{pmatrix} I \\ zF(z) \\ G(z) \end{pmatrix} = \begin{pmatrix}
  U(z) \\ F(z) \\ zG(z) \end{pmatrix}.
\end{equation}
Multiplying both sides of this equation by $X(z,w) =
\begin{pmatrix} p(z,w) \\wE(z,w) \end{pmatrix}$ gives
\begin{equation} \label{V}
V \begin{pmatrix} p(z,w) \\ wE(z,w) \\ zF(z)X(z,w) \\ G(z)
  X(z,w) \end{pmatrix} 
= \begin{pmatrix} \rev{p}(z,w) \\ E(z,w) \\ F(z)X(z,w) \\ zG(z)
  X(z,w) \end{pmatrix}.
\end{equation}
Let $A(z,w) = F(z)X(z,w)$ and $B(z,w) = G(z)X(z,w)$.  The entries of
$A$ and $B$ are linearly independent, because if $v_1 \in \C^{n_1},
v_2 \in \C^{n_2}$ and
\[
0 \equiv v_1A(z,w)+v_2B(z,w) = (v_1F(z)+v_2G(z))X(z,w)
\]
then since
\[
X(z,w) = \begin{pmatrix} p_0(z) & \begin{matrix} p_1(z) & \cdots &
    p_m(z) \end{matrix} \\
0 & E(z) \end{pmatrix} \begin{pmatrix} 1 \\ w \\ \vdots
  \\ w^m \end{pmatrix}
\]
we have
\[
0 \equiv  (v_1 F(z)+v_2 G(z))\begin{pmatrix} p_0(z) & \begin{matrix} p_1(z) & \cdots &
    p_m(z) \end{matrix} \\
0 & E(z) \end{pmatrix}. 
\]
The matrix on the right is invertible in $\D$ except at possible zeros
of $p_0$, so we get $v_1 F(z)+v_2 G(z) \equiv 0$ which implies
$v_1=0$ and $v_2=0$.  

 Taking the norm squared of both sides of \eqref{V} gives the
 following formula since $V$ is a unitary
\[
\begin{aligned}
& |p(z,w)|^2 +|w|^2|E(z,w)|^2 + |z|^2|A(z,w)|^2 + |B(z,w)|^2\\
&= |\rev{p}(z,w)|^2 + |E(z,w)|^2 + |A(z,w)|^2 + |z|^2|B(z,w)|^2.
\end{aligned}
\]
If we rearrange we get the desired sum of squares formula
\[
|p(z,w)|^2 - |\rev{p}(z,w)|^2 = (1-|w|^2)|E(z,w)|^2 +
(1-|z|^2)(|A(z,w)|^2 - |B(z,w)|^2).
\]
One final technicality is that while $E$ has entries that are
polynomials, it is not clear that the same holds for $A$ and $B$.  To
show they are polynomials we go through a longer process of proving
the following ``transfer function'' representation which is
interesting in its own right.

Set
\begin{equation} \label{delgam}
\begin{aligned}
\Delta(z,w) &= \begin{pmatrix} wI_{m}& 0 & 0 \\ 0 & zI_{n_1} & 0
  \\ 0 & 0 & I_{n_2} \end{pmatrix} \\ 
\Gamma(z,w) &= \begin{pmatrix}
  I_{m}& 0 & 0 \\ 0 & I_{n_1} & 0 \\ 0 & 0 & zI_{n_2} \end{pmatrix}.
\end{aligned}
\end{equation}

\begin{theorem} Suppose $p \in \C[z,w]$ has no zeros in $\T\times\D$,
  degree $(n,m)$, and no factors in common with $\rev{p}$. Then, there
  exists a $(1+m+n)\times (1+m+n)$ unitary matrix $V$ such that
\begin{equation} \label{transferrep}
\frac{\rev{p}(z,w)}{p(z,w)} = V_{1} + V_{2} \Delta(z,w)
(\Gamma(z,w) - V_{4} \Delta(z,w))^{-1} V_{3}.
\end{equation}
Here we use the block form indicated in \eqref{Vblocks} and again
$n_2$ is the number of zeros of $p(z,0)$ in $\D$ and $n_1=n-n_2$.
\end{theorem}

A technicality we must address is whether the matrix we invert above
is non-degenerate.  The fact that the rows of $F$ and $G$ are linearly
independent is used to show this.

By \eqref{V}, the map sending
\[
\begin{pmatrix} p(z,w) \\ wE(z,w) \\ zA(z,w) \\  B(z,w)  \end{pmatrix}
\mapsto \begin{pmatrix} \rev{p}(z,w) \\ E(z,w) \\ A(z,w)
  \\ zB(z,w) \end{pmatrix}
\]
extends to the unitary $V$.  Then,
\[
\begin{aligned}
V_{1} p + V_{2} \Delta(z,w)
\begin{pmatrix} E \\ A
  \\ B \end{pmatrix}  &= \rev{p} \\
V_{3} p + V_{4} \Delta(z,w)
 \begin{pmatrix} E \\ A
  \\ B \end{pmatrix} &= \Gamma(z,w) \begin{pmatrix} E \\ A
  \\ B \end{pmatrix} 
\end{aligned}
\]
which implies
\[
pV_{3} = \left( \Gamma(z,w) - V_{4} \Delta(z,w)
\right) \begin{pmatrix} E \\ A \\ B \end{pmatrix}
\]
We would like to invert the matrix on the right, so we need to make
sure 
\begin{equation} \label{determinant}
\det \left(\Gamma(z,w) - V_{4} \Delta(z,w)\right)
\end{equation}
is not identically zero.  This is equivalent to 
\[
\det \left(\begin{pmatrix} wI_{m}&
  0 & 0 \\ 0 & zI_{n_1}
  & 0 \\ 0 & 0 & z^{-1} I_{n_2} \end{pmatrix} - V_{4} \right) 
\]
being non-trivial by simple matrix manipulations. The coefficient of
$w^m$ will occur as 
\[
\det\left(\begin{pmatrix} zI_{n_1} & 0 \\ 0 &
  z^{-1} I_{n_2} \end{pmatrix} - V'_4 \right)
\]
where $V_4'$ is the lower right $n\times n$ block of $V$ as in
\eqref{Vblocks}.  We shall show this determinant is non-vanishing for
$z \in \T$.  If it does vanish for some $z=\zeta \in \T$, then there
exists a nonzero $v=(v_1,v_2)\in \C^{n} =\C^{n_1+n_2}$ such that
\[
(v_1,v_2)\begin{pmatrix} \zeta I_{n_1} & 0 \\ 0 &
  \bar{\zeta} I_{n_2} \end{pmatrix}  = (v_1,v_2) V'_4.
\]
This implies that $\|v\| =\|v V'_4\|$ and since $V$ is a unitary $v
V_3' = 0$.  Then, by \eqref{VFG}
\[
(0,v_1, v_2) V \begin{pmatrix} I \\ zF(z) \\ G(z) \end{pmatrix} = 
(0,\zeta v_1, \bar{\zeta} v_2) \begin{pmatrix} I \\ zF(z)
  \\ G(z) \end{pmatrix} 
= (0,v_1,v_2) \begin{pmatrix}
  U(z) \\ F(z) \\ zG(z) \end{pmatrix} 
\]
so that
\[
v_1 \zeta zF(z)+v_2 \bar{\zeta} G(z) = v_1 F(z) + v_2 zG(z)
\]
and then
\[
(z\zeta -1)v_1F(z) + (\bar{\zeta}- z)v_2 G(z) \equiv 0.
\]
This implies 
\[
v_1 F(z) - \bar{\zeta} v_2 G(z) \equiv 0
\]
contradicting \eqref{detail}.  Therefore, the determinant in
\eqref{determinant} is not identically zero, and  
\begin{equation} \label{solving}
p(z,w) \left(\Gamma(z,w) - V_{4} \Delta(z,w) \right)^{-1} V_{3}
= \begin{pmatrix} E(z,w) \\ A(z,w) \\ B(z,w) \end{pmatrix}
\end{equation}
which in turn yields \eqref{transferrep}.  Examining
\eqref{transferrep} we see that since $p$ has degree $(n,m)$, since
\[
\det \left(\begin{pmatrix} I_{m}&
  0 & 0 \\ 0 & I_{n_1}
  & 0 \\ 0 & 0 & zI_{n_2} \end{pmatrix} - V_{4} \begin{pmatrix} wI_{m}&
  0 & 0 \\ 0 & zI_{n_1}
  & 0 \\ 0 & 0 & I_{n_2} \end{pmatrix}\right)
\]
has degree at most $(n,m)$, and since $p$ and $\rev{p}$ have no common
factors, we must have that $p$ is a constant multiple of the above
determinant else \eqref{transferrep} could be reduced further.  This
implies that the left hand side of \eqref{solving} is a vector
polynomial and finally we see that the entries of $E,A,B$ are
polynomials.  By Cramer's rule the entries of $E$ have degree at most
$(n,m-1)$ and the entries of $A$ and $B$ have degree at most
$(n-1,m)$.

\subsubsection{Unitary valued rational functions on the circle}
The following is undoubtedly well-known material from systems theory; 
however, we were unable to find a suitable reference so we include a
detailed explanation.

\begin{theorem}[Smith Normal form \cite{HK}] \label{SNF}
Let $R$ be a principal ideal domain and let $A$ be an $N\times N$
matrix with entries in $R$.  There exists a unique (up to units)
diagonal matrix $D \in R^{N\times N}$, called the Smith Normal form,
with entries $D_1 | D_2 | \dots | D_N$ such that
\[
A =  SDT
\]
where $S,T \in R^{N\times N}$ and $S^{-1}, T^{-1} \in R^{N\times N}$.
The matrix $S$ is formed through the row operations of (1) multiplying a
row by an element of $R$ and adding the result onto another row, (2)
switching two rows, and (3) multiplying a row by a unit in $R$.
 
The entries $D_j$ may also be computed as
\[
D_j = \frac{\gcd_j(A)}{\gcd_{j-1}(A)}
\]
where $\gcd_j(A)$ represents the greatest common divisor of
determinants of all $j\times j$ submatrices of $A$ ($\gcd_0:=1$).
\end{theorem}

Let $U \in \C(z)^{N\times N}$ be a rational $N\times N$ matrix
function of one variable which is unitary valued on the unit circle.
Let $R$ be the ring of fractions $\C[z]\mathcal{S}^{-1}$ where
$\mathcal{S}$ is the multiplicative set $\mathcal{S} = \{q \in \C[z]:
q(z) \ne 0 \text{ for } z \in \D\}$.  We may write $U = \frac{1}{q}Q$
where $q \in \C[z]$ has all zeros in $\D$ and $Q \in R^{N\times N}$.
Let $D$ be the Smith Normal form of $Q$ in $R$.  Write
\[
\frac{D_j}{q} = \frac{d_j}{q_j}
\]
in lowest terms and define
\[
N_1 = \sum_{j=1}^{N} \text{\# zeros of } d_j \text{ in } \D \text{
  counting multiplicity}
\]
\[
N_2 = \sum_{j=1}^{N} \text{\# zeros of } q_j \text{ in } \D \text{
  counting multiplicity}.
\]

\begin{theorem} \label{Udecomp} 
With $U$ as above, there exist an $N_1\times N$ matrix function $F$
and an $N_2\times N$ matrix function $G$ such that
\[
\frac{I-U(\zeta)^*U(z)}{1-z\bar{\zeta}} = F(\zeta)^*F(z) -
G(\zeta)^*G(z).
\]
The rows of $F$ and $G$ together form a linearly independent set of
vector functions on $\D$.
\end{theorem}

\begin{proof} 
We shall use $\vec{H}^2$ to denote the vector valued Hardy space
$H^2(\T)\otimes \C^N$ for short.  Now $U\vec{H}^2$ is a reproducing
kernel Hilbert space with point evaluations in $\D$ except at the
poles of $U$ in $\D$.  In $\vec{L}^2=L^2(\T) \otimes \C^N$, if $f \in
\vec{H}^2$ and $v \in \C^N$
\[
\ip{Uf}{ \frac{U U(\zeta)^*}{1-\cdot \bar{\zeta}}v}_{L^2} =
  \ip{f}{\frac{U(\zeta)^{*} v}{1-\cdot\bar{\zeta}}}_{L^2} =
    \ip{U(\zeta)f(\zeta)}{v}_{\C^N}
\]
 which shows $U\vec{H}^2$ has reproducing kernel 
\[
\frac{U(z)U(\zeta)^*}{1-z\bar{\zeta}}.
\]
Notice $U\vec{H}^2$ is not necessarily contained in $\vec{H}^2$. The
space $U\vec{H}^2 \vee \vec{H}^2$ is therefore a reproducing kernel
Hilbert space containing both spaces.

 Consider the kernel
\[
K(z;\zeta)= \frac{I-U(z)U(\zeta)^*}{1-z\bar{\zeta}}
\]
which is not necessarily positive definite but is rather the
difference of two reproducing kernels
\[
K = K_{\vec{H}^2} - K_{U\vec{H}^2}.
\]
We shall in general use $K_\mcH$ to denote the reproducing kernel of a
space $\mcH$ in $U\vec{H}^2\vee \vec{H}^2$.  We can decompose
$\vec{H}^2 = (\vec{H}^2\cap U\vec{H}^2) \oplus (\vec{H}^2
\ominus(\vec{H}^2\cap U\vec{H}^2))$ so that
\[
K_{\vec{H}^2} = K_{\vec{H}^2\cap U\vec{H}^2} + K_{\vec{H}^2
  \ominus(\vec{H}^2\cap U\vec{H}^2)}
\]
and similarly 
\[
K_{U\vec{H}^2} = K_{\vec{H}^2\cap U\vec{H}^2} + K_{U\vec{H}^2
  \ominus(\vec{H}^2\cap U\vec{H}^2)}.
\]
Therefore, 
\[
K = K_{\vec{H}^2 \ominus(\vec{H}^2\cap U\vec{H}^2)} - K_{U\vec{H}^2
  \ominus(\vec{H}^2\cap U\vec{H}^2)}
\]
 The spaces $\vec{H}^2 \ominus(\vec{H}^2\cap U\vec{H}^2), U\vec{H}^2
 \ominus(\vec{H}^2\cap U\vec{H}^2)$ are actually finite dimensional.
 We can compute their dimensions as follows.

Write $U = \frac{1}{q}Q$ where $q$ has all zeros in $\D$ and $Q$ has
entries in $R$.  Note that since $U$ is unitary on the circle, $U$ has
no poles on the circle ($U$ is bounded near any potential
singularities).  Therefore, the entries of $Q$ belong to the smaller
ring $R_0 = \C[z]\mathcal{S}_0^{-1}$ where $\mathcal{S}_0$ is the
multiplicative set $\mathcal{S}_0 = \{q \in \C[z]: q(z) \ne 0 \text{
  for } z \in \cD\}$.  By Theorem \ref{SNF}, we may write $Q = SDT$
where $S,T$ are matrices with entries in $R_0$ whose inverses have the
same property, and $D$ is the Smith Normal form of $Q$ in $R_0$.  The
elements $D_j\in R_0$ have no zeros on the unit circle since $\det U =
\det Q/q^N = \det S\det T \prod (D_j/q)$ has no zeros on $\T$ and $q$ has no zeros
on $\T$.  So, $Q$ has the same Smith Normal form $D$ in $R$ by the gcd
characterization of the Smith Normal form.  

 Now, $\vec{H}^2 \ominus (\vec{H}^2\cap U\vec{H}^2)$ is isomorphic as
 a vector space to the quotient $\vec{H}^2/(\vec{H}^2\cap
 U\vec{H}^2)$, and since $T\vec{H}^2 = \vec{H}^2 = S\vec{H}^2$ we see
 that
\[
\vec{H}^2/(\vec{H}^2\cap U\vec{H}^2) \cong \vec{H}^2/(\vec{H}^2 \cap
\frac{1}{q}D \vec{H}^2)
\]
which breaks up into the algebraic direct sum of the spaces 
\[
H^2/ (H^2\cap\frac{D_j}{q} H^2).
\]
  Any zeros of $D_j$ or $q$ in $\C\setminus \D$ can be
absorbed into $H^2$ so that $\frac{D_j}{q}H^2 = \frac{d_j}{q_j} H^2$
for some $d_j$ and $q_j$ with all zeros in $\D$ and no common zeros
(after canceling).  The space
\[
H^2 \cap \frac{d_j}{q_j} H^2 = d_j H^2
\]
and $H^2/d_jH^2$ has dimension equal to the number of zeros of $d_j$
in $\D$. Therefore, $\vec{H}^2/(\vec{H}^2\cap U\vec{H}^2)$ has
dimension equal to
\[
N_1 = \sum_{j=1}^{N} \text{\# zeros of } d_j \text{ in } \D \text{ counting
  multiplicity.}
\]

A similar analysis shows that $U\vec{H}^2/(\vec{H}^2\cap U\vec{H}^2)$
has dimension equal to
\[
N_2 = \sum_{j=1}^{N} \text{\# zeros of } q_j \text{ in } \D \text{
  counting multiplicity}.
\]
If $\{f_1,\dots, f_{N_1}\}$ is an orthonormal basis for
$\vec{H}^2\ominus (\vec{H}^2\cap U\vec{H}^2)$ and $\{g_1,\dots,
g_{N_2}\}$ is an orthonormal basis for $U\vec{H}^2\ominus
(\vec{H}^2\cap U\vec{H}^2)$ then
\[
K(z;\zeta) = \sum f_j(z)f_j(\zeta)^* - \sum g_j(z) g_j(\zeta)^*
\]
which if we form an $N\times N_1$ matrix $F=(f_1, \dots, f_{N_1})$ and
an $N\times N_2$ matrix $G = (g_1,\dots, g_{N_2})$ can rewrite as
\[
\frac{I-U(z)U(\zeta)^*}{1-z\bar{\zeta}} = F(z)F(\zeta)^* -
G(z)G(\zeta)^*.
\]
Since $\vec{H}^2 \ominus (\vec{H}^2\cap U\vec{H}^2)$ and $U\vec{H}^2
\ominus (\vec{H}^2\cap U\vec{H}^2)$ have trivial intersection, the
columns of $F$ and $G$ form an independent set of vector functions.

Of course, applying the above work to $U^t$ instead would yield a
formula of the form 
\[
\frac{I-U(\zeta)^*U(z)}{1-z\bar{\zeta}} = F(\zeta)^*F(z) -
G(\zeta)^*G(z)
\]
after switching $z$ and $\zeta$ and taking conjugates, where now $F$
and $G$ are $N_1\times N$ and $N_2\times N$ valued respectively. Note
we are not saying they are the same $F$ and $G$ as before, but the
dimensions $N_1$ and $N_2$ are preserved because the transpose does
not change the diagonal term in the Smith Normal form.  The rows of
$F$ and $G$ form an independent set of vector functions just as
above. 
\end{proof}

\subsubsection{A particular choice of $U$} \label{Uproperties}
Recall the matrix function $U$ from Section \ref{sec:kummert}
\[
U(z) =  \begin{pmatrix} 
\begin{matrix} \rev{p}_m(z) & \cdots & \rev{p}_1(z) \end{matrix} 
& \rev{p}_0(z) \\ 
E(z) & 0 \end{pmatrix} \begin{pmatrix} p_0(z) & \begin{matrix} p_1(z) & \cdots &
    p_m(z) \end{matrix} \\
0 & E(z) \end{pmatrix}^{-1}.
\]
Note
\[
U = \frac{1}{p_0}Q = \frac{1}{p_0} \begin{pmatrix}
\begin{matrix} \rev{p}_m & \cdots & \rev{p}_1 \end{matrix} 
& \rev{p}_0 \\ 
E & 0 \end{pmatrix}
\begin{pmatrix} 1 & -\begin{pmatrix} p_1 & \cdots &
  p_m \end{pmatrix} E^{-1} \\
0 & p_0 E^{-1} \end{pmatrix}
\]
where $Q$ has entries in $R$ since $\det E$ may have zeros on $\T$
while $p_0(z)=p(z,0)$ has no zeros on $\T$.  We now show how to
compute the Smith Normal form of $Q$ in $R$.

Now,
\[
Q = \begin{pmatrix} 1 & 0 \\ 0 &
  E \end{pmatrix} \begin{pmatrix} \begin{matrix} \rev{p}_m & \cdots
    & \rev{p}_1 \end{matrix} & \rev{p}_0 \\ I &
  0 \end{pmatrix} \begin{pmatrix} 1 & -\begin{pmatrix} p_1 & \cdots &
  p_m \end{pmatrix}  \\
  0 & p_0I \end{pmatrix} \begin{pmatrix} 1 & 0 \\ 0 & E^{-1} \end{pmatrix}.
\]
It is not hard to see that the product of the inner two matrices can
be converted to the diagonal matrix $D$ with entries $1 ,p_0,\dots,
p_0, p_0\rev{p}_0$ using row and column operations in $R$, which is
the Smith Normal form of $Q$.  The entries of $\frac{1}{p_0}D$ are
then $1/p_0,1,\dots, 1, \rev{p}_0$. 

This proves that for $n_2$ equal to the number of zeros of $p_0$ in
$\D$ and $n_1 = n-n_2$
\[
\frac{I-U(\zeta)^* U(z)}{1-\bar{\zeta}z} = F(\zeta)^*F(z) -
G(\zeta)^*G(z)
\]
where $F$ is $n_1 \times (m+1)$ and $G$ is $n_2\times (m+1)$.

\subsection{Generalized distinguished varieties and determinantal representations} \label{detreps}

Distinguished varieties are a class of curves introduced in
Agler-McCarthy \cite{AMdv} because they play a natural role in
multivariable operator theory and function theory on the bidisk (see
\cite{gKdv}, \cite{JKMpick}, \cite{AKMisopairs}).  The zero set $Z_p$
of a polynomial $p \in \C[z,w]$ is a distinguished variety if
\[
Z_p \subset \D^2 \cup \T^2 \cup \E^2
\]
where $\E = \C \setminus \cD$. Notice the curve $Z_p \cap \D^2$ exits
the boundary of $\D^2$ through the distinguished boundary $\T^2$;
hence the name \emph{distinguished} variety. This area
    is part of a larger topic of understanding algebraic curves and
    their interaction with $\T^2$.  See \cite{AMS06}, \cite{AMS08}.

The sums of squares theorem, Theorem \ref{sosthm} or Theorem
\ref{gensosthm}, naturally leads to the study of a more general class
of curves using the methods of \cite{gKdv}.  We say that the zero set
$Z_p$ of $p \in \C[z,w]$ is a \emph{generalized distinguished variety}
if it satisfies
\begin{equation} \label{weakdv}
Z_p \subset (\D\times \C) \cup \T^2 \cup (\E \times \C) \text{ or }
\end{equation}
\[
Z_p \subset (\C\times \D) \cup \T^2 \cup (\C \times \E).
\]
That is, $Z_p$ does not intersect the (relatively small) set
$(\T\times \D) \cup (\T \times \E)$ in the former case above.  We
shall show that generalized distinguished varieties share much of the
structure of distinguished varieties, and in particular they possess a
determinantal representation generalizing one of the main theorems in
\cite{AMdv}.  We use the notation \eqref{delgam} below.

\begin{theorem} \label{thm:detrep}
Suppose $p\in \C[z,w]$ has degree $(n,m)$, has no factors involving
$z$ alone, and satisfies \eqref{weakdv}. Then, there exists an
$(m+n)\times (m+n)$ unitary matrix $U$ such that $p$ is a constant
multiple of
\[
\det\left(U  \Delta(z,w)  - \Gamma(z,w)
\right),
\]
where $n_2$ is the number of zeros of $p(z,0)$ in $\D$ and $n_1 =
n-n_2$.

\end{theorem}

If $Z_p$ is a distinguished variety then $n_2=n$ and we get the
representation
\[
Z_p = \left\{(z,w): \det\left(U \begin{pmatrix} wI_m& 0 \\ 0 &
  I_n \end{pmatrix} - \begin{pmatrix} I_m & 0 \\ 0 &
  zI_{n} \end{pmatrix}\right)=0\right\}.
\]
If we write $U = \begin{pmatrix} A & B \\ C & D \end{pmatrix}$, then
the above zero set can be written as 
\[
\det(\Phi(w) - zI_n) = 0
\]
in terms of the matrix rational inner function
\[
\Phi(w) = D + wC(I-wA)^{-1} B
\]
at least outside of the poles of $\Phi$.  This is how the
characterization of distinguished varieties is stated in \cite{AMdv}.

\begin{lemma} 
Suppose $p\in \C[z,w]$ has degree $(n,m)$, has no factors involving
$z$ alone, and satisfies \eqref{weakdv}.  Then, $p=\mu \rev{p}$ for
some $\mu \in \T$.
\end{lemma}


\begin{proof}
For each $z \in \T$, $p(z,\cdot)$ has all zeros in $\T$---as does
$\rev{p}(z,\cdot)$.  Since $z \in \T$, we see that $p(z,\cdot)$ and
$\rev{p}(z,\cdot)$ have the same roots (counting multiplicity, since
they approach zero at the same rate near a root because
$|p|=|\rev{p}|$ on $\T^2$).  Therefore, if we write 
\[
p(z,w) = \sum_{j=0}^{m} p_j(z)w^j, \qquad \rev{p}(z,w) = \sum_{j=0}^{m}
\rev{p}_{m-j}(z) w^j
\]
then for all $z \in \T$,  $\rev{p}_m(z)p(z,\cdot) =
p_0(z)\rev{p}(z,\cdot)$ since these polynomials have the same roots
and same leading coefficient.  Hence, $\rev{p}_m(z)p(z,w) = p_0(z)
\rev{p}(z,w)$ for all $(z,w) \in \C^2$.  By assumption $p$ has no
factors involving $z$ alone, and therefore $p$ divides $\rev{p}$.
Similarly $\rev{p}$ divides $p$, so that $p = C\rev{p}$ for some
constant $C$.  Since $|p| = |\rev{p}|$ on $\T^2$, $C$ must be
unimodular.
\end{proof}

Because of this lemma we can assume $p=\rev{p}$ by replacing $p$ with
an appropriate constant multiple.

\begin{lemma} 
Suppose $p=\rev{p} \in \C[z,w]$ has degree $(n,m)$ satisfies
\eqref{weakdv} and is irreducible.  Then,
\begin{itemize}
\item $mp = \rev{\frac{\partial p}{\partial w}} + w
\frac{\partial p}{\partial w}$ and
\item $\rev{\frac{\partial p}{\partial w}}$ has no zeros in $\T\times
  \D$ and no factors in common with $\frac{\partial p}{\partial w}$.
\end{itemize}
We reflect $\partial p/\partial w$ at the degree $(n,m-1)$.
\end{lemma}


\begin{proof}
The identity $mp = \rev{\frac{\partial p}{\partial w}} + w
\frac{\partial p}{\partial w}$ is straightforward assuming $p =
\rev{p}$.

For $t<1$, let $p_t(z,w) = p(z,tw)$.  Then, $p_t$ has no
zeros in $\T\times \cD$ and
\[
|p_t(z,w)|^2 - |\rev{p_t}(z,w)|^2 \geq 0
\]
for $(z,w) \in \T\times \cD$.  Therefore,
\[
\lim_{t\nearrow 1} \frac{|p_t(z,w)|^2 - |\rev{p_t}(z,w)|^2}{1-t^2}
\geq 0
\]
but the above limit equals
\[
m|p(z,w)|^2 -2\text{Re}(w\frac{\partial p}{\partial w}
\overline{p(z,w)}) 
\geq 0
\]
for $(z,w) \in \T\times \cD$. Now, since $mp = \rev{\frac{\partial
    p}{\partial w}} + w \frac{\partial p}{\partial w}$
\[
m^2|p(z,w)|^2 -2m\text{Re}(w\frac{\partial p}{\partial w}
\overline{p(z,w)}) = |\rev{\frac{\partial p}{\partial w}}|^2 - |w
\frac{\partial p}{\partial w}|^2 \geq 0.
\]
Therefore, any zero of $\rev{\frac{\partial p}{\partial w}}$ in
$\T\times \cD$ is a zero of $w\frac{\partial p}{\partial w}$ and hence
will be a zero of $p$, which by assumption has no zeros in $\T\times
\D$.  So, $\rev{\frac{\partial p}{\partial w}}$ has no zeros in
$\T\times \D$.  

Now, $\rev{\frac{\partial p}{\partial w}}$ can have no factors in
common with $\frac{\partial p}{\partial w}$, else $w\frac{\partial
  p}{\partial w}$ and $p$ have a common factor.  As $p$ is assumed to
be irreducible, this is impossible.  
\end{proof}

\begin{proof}[Proof of Theorem \ref{thm:detrep}]
It is sufficient to prove the theorem for irreducible $p=\rev{p}$
since we can write the determinantal representation in terms of blocks
corresponding to each irreducible factor of $p$.  With this assumption
$\rev{\frac{\partial p}{\partial w}}$ has no zeros on $\T\times \D$
and no factors in common with $w\frac{\partial p}{\partial w}$, and
$mp(z,0)=\rev{\frac{\partial p}{\partial w}}(z,0)$ has $n_2$ zeros in
$\D$. The proof of Theorem \ref{gensosthm} says there are vector
polynomials $A\in \C^m[z,w],B \in \C^{n_1}[z,w],C \in \C^{n_2}[z,w]$
such that
\[
\rev{\frac{\partial p}{\partial w}}(z,w)\overline{\rev{\frac{\partial
      p}{\partial w}}(\zeta,\eta)} - w\bar{\eta} \frac{\partial
  p}{\partial w}(z,w) \overline{\frac{\partial p}{\partial
    w}(\zeta,\eta)}
\]
equals
\[
(1-w\bar{\eta})A(\zeta,\eta)^*A(z,w) +
(1-z\bar{\zeta})(B(\zeta,\eta)^*B(z,w)- C(\zeta,\eta)^*C(z,w)).
\]
By Theorem \ref{gensosthm} we can choose
$A,B,C$ so that $A$ has degree at most $(n,m-1)$ while $B,C$ have
degree at most $(n-1,m)$.  Furthermore, the entries of $B$ and $C$
together form a linearly independent set of polynomials.

The identity $mp=\rev{\frac{\partial p}{\partial w}} + w
\frac{\partial p}{\partial w}$ proves
\[
m^2p\bar{p} -m(w\frac{\partial p}{\partial w}
\bar{p})-mp\overline{\eta\frac{\partial p}{\partial w}} =
\rev{\frac{\partial p}{\partial w}}\overline{\rev{\frac{\partial
      p}{\partial w}}} - w\bar{\eta} \frac{\partial p}{\partial w}
\overline{\frac{\partial p}{\partial w}}.
\]

On the zero set $Z_p$ we get the formula
\[
0 = (1-w\bar{\eta})A(\zeta,\eta)^*A(z,w) +
(1-z\bar{\zeta})(B(\zeta,\eta)^* B(z,w) - C(\zeta,\eta)^*C(z,w)).
\]
A lurking isometry argument now produces the formulas we want.  First,
we rearrange
\[
\begin{aligned}
&w\bar{\eta} A(\zeta,\eta)^*A(z,w) + z\bar{\zeta} B(\zeta,\eta)^*
B(z,w) + C(\zeta,\eta)^* C(z,w) \\
&=  A(\zeta,\eta)^*A(z,w) + B(\zeta,\eta)^*
B(z,w) + z\bar{\zeta}C(\zeta,\eta)^* C(z,w)
\end{aligned}
\]
for $(z,w),(\zeta,\eta) \in Z_p$.  Then, the map
\[
\begin{pmatrix} w A(z,w) \\ zB(z,w) \\ C(z,w) \end{pmatrix} \mapsto 
\begin{pmatrix}  A(z,w) \\ B(z,w) \\ zC(z,w) \end{pmatrix}
\]
extends to a well-defined unitary on the span of the elements on the
left (as $(z,w)$ varies over $Z_p$) to the span of the elements on the
right.  Since the ambient spaces have the same dimension we can extend
to an $(m+n)\times (m+n)$ unitary $U$ such that
\[
U\begin{pmatrix} w A(z,w) \\ zB(z,w) \\ C(z,w) \end{pmatrix} = 
\begin{pmatrix}  A(z,w) \\ B(z,w) \\ zC(z,w) \end{pmatrix}
\]
on $Z_p$. 
 Then,
\[
\left(U \Delta(z,w) - \Gamma(z,w)\right) \begin{pmatrix} A \\ B
  \\ C \end{pmatrix} = 0
\]
and since $A,B,C$ vanish at only finitely many points in $Z_p$, we get
\begin{equation} \label{maindet}
\det\left(U \Delta(z,w) - \Gamma(z,w) \right) = 0.
\end{equation}
The polynomial on the left has degree less than or equal to that of
$p$ and vanishes on $Z_p$.  Since $p$ is irreducible, it must either
be a nonzero multiple of $p$ or it must be identically zero.

Claim: The determinant in \eqref{maindet} is not identically zero.

The explanation is similar to before.  If the determinant is
identically zero, then by simple matrix manipulations
\[
\det \left(\begin{pmatrix} wI_m& 0 & 0 \\ 0 & zI_{n_1} & 0 \\ 0 & 0 &
  z^{-1}I_{n_2} \end{pmatrix}-U\right) \equiv 0.
\]
The coefficient of $w^m$ is 
\[
\det(\begin{pmatrix} zI_{n_1} & 0 \\ 0 &
  z^{-1}I_{n_2} \end{pmatrix} - U_4 ) \equiv 0
\]
where $U_4$ is the lower-right $n\times n$ block of $U$.  The above
determinant cannot vanish for any $z=\zeta \in \T$, since if it does
there exists a non-zero vector $v=(v_1,v_2)$ such that
\[
(\zeta v_1,\bar{\zeta} v_2) = (v_1,v_2) U_4.
\]
Then, on $Z_p$
\[
(0,\zeta v_1,\bar{\zeta} v_2)  \begin{pmatrix} w A(z,w) \\ zB(z,w)
  \\ C(z,w) \end{pmatrix} = (0,v_1,v_2) U \begin{pmatrix} w A(z,w) \\ zB(z,w)
  \\ C(z,w) \end{pmatrix} = (0,v_1,v_2) \begin{pmatrix}  A(z,w) \\ B(z,w) \\ zC(z,w) \end{pmatrix} 
\]
and so $\zeta z v_1 B(z,w) + \bar{\zeta} v_2C(z,w) = v_1 B(z,w) + z
v_2C(z,w)$ which implies 
\[
(\zeta z -1)v_1 B(z,w) + (\bar{\zeta} -z)v_2 C(z,w)  = 0
\]
which in turn implies
\[
\zeta v_1 B(z,w) - v_2 C(z,w) = 0 \text{ on } Z_p
\]
since $z=\bar{\zeta}$ for finitely many $(z,w) \in Z_p$.  Since $p$ is
irreducible, $p$ divides $\zeta v_1 B - v_2 C$.  Since $B$ and $C$
have degree at most $(n-1,m)$, we see that $\zeta v_1 B - v_2 C=0$,
which is not possible unless $v_1$ and $v_2$ are zero vectors, which
they are not.

So, the determinant in \eqref{maindet} is not identically zero.  It
follows that $p$ is a multiple of the determinant in \eqref{maindet}.

\end{proof}

\end{document}